\documentclass[12pt]{article}

\usepackage{geometry}
\usepackage{amsfonts}
\usepackage{amsmath}
\usepackage{amssymb}
\usepackage{amsthm}
\usepackage{verbatim}
\usepackage{graphicx}
\usepackage{mathrsfs}
\usepackage{url}
\usepackage{bbm}
\usepackage{booktabs}
\usepackage[title]{appendix}

\newdimen\abovetopsep
\newtheorem{thm}{Theorem}
\numberwithin{thm}{section}

\newtheorem{lem}[thm]{Lemma}
\newtheorem{prop}[thm]{Proposition}

\theoremstyle{remark}
\newtheorem*{rem}{Remark}
\newtheorem{remn}[thm]{Remark}

\numberwithin{equation}{section}

\usepackage{mypackage}
\usepackage[linktocpage=true]{hyperref}

\begin{document}
\title{The domino shuffling height process and its hydrodynamic limit}
\author{Xufan Zhang\\ Brown University\\ Providence, RI 02912, USA }
\date{}

\maketitle
\begin{abstract}
The famous domino shuffling algorithm was invented to generate the domino tilings of the Aztec Diamond. Using the domino height function, we view the domino shuffling procedure as a discrete-time random height process on the plane. The hydrodynamic limit from an arbitrary continuous profile is deduced to be the unique viscosity solution of a Hamilton-Jacobi equation $u_t+H(u_x)=0$, where the determinant of the Hessian of $H$ is negative everywhere. The proof involves interpolation of the discrete process and analysis of the limiting semigroup of the evolution. In order to identify the limit, we use the theories of  dimer models as well as Hamilton-Jacobi equations.

It seems that our result is the first example in $d>1$ where such a full hydrodynamic limit with a nonconvex Hamiltonian can be obtained for a discrete system. We also define the shuffling height process for more general periodic dimer models, where we expect similar results to hold.
\end{abstract}
\tableofcontents
\section{Introduction}

The dimer model, which consists of the weighted perfect matchings on graphs, is a well-studied model in mathematical physics. The domino model  is the dimer model on a (possibly infinite) region of the $\mbb Z^2$ lattice. Alternatively, it can be thought of as tiling a region of the square grid exactly by $1\times 2$ and $2\times 1$ dominoes. For a comprehensive survey on the subject of dimer models, see \cite{RK09}. Necessary knowledge will be reviewed in the paper.

The domino shuffling is a discrete-time random dynamics operating on the domino model, introduced originally by Elkies et al. in \cite{EKLP92} to compute the generating function of domino tilings of a specific family of graphs called Aztec Diamonds with periodic weights. In particular, the domino shuffling provides a simple algorithm to generate the uniform measure of tilings of Aztec Diamonds and led to the first rigorous proof of the famous Artic Circle Theorem by Jockusch et al.(\cite{JPS98}), that describes the asymptotic shape of the central sub-region of a random tiling.

The domino shuffling seemed quite mysterious the way it was presented in \cite{EKLP92} initially. Propp(\cite{JP03}) later gave a much clearer explanation using a procedure called urban renewal, or spider move, which also allows natural generalizations to different graphs. This is the approach we will take to define the shuffling dynamics in Section~\ref{sec:setup}. As it turns out, this also naturally converts the shuffling dynamics into a random height process, which we shall call the {\it shuffling height process}. The height process is a $(2+1)$-dimensional evolution which is discrete in both space (2-dimensional) and time (1-dimensional). The purpose of this paper is to prove that, when rescaling both space and time parameters by $n$, and starting nearby an arbitrary continuous profile (subject to certain legality requirement), as $n\rar \infty$, the height process evolves according to a first-order, nonlinear Hamilton-Jacobi equation
\ba{u_t+H(u_x)=0 \label{intropde}}
where, in the case of uniform shuffling, the Hamiltonian is given by
\al{
H(\rho_1,\rho_2)=\fr4{\pi}\cos^{-1}\lp \fr{1}{2}\cos\lp\fr{\pi \rho_2}2\rp - \fr{1}{2}\cos\lp\fr{\pi \rho_1}2\rp\rp.
}
The convergence is uniform in any compact subset of the spacetime. A subtlety is that the PDE develops shocks even starting from a smooth profile, so we need to consider a specific weak solution called the viscosity solution.

Before sketching the proof, we would like to mention some other works on domino shuffling. Johansson(\cite{KJ05}) and Nordenstam(\cite{EN10}) related the domino shuffling on the Aztec Diamond to a determinantal point process. As a result, they were able to prove that, under appropriate rescaling, the boundary of the arctic circle converges to the Airy process, and the turning point converges to the GUE minor process. In 2014, Borodin and Ferrari(\cite{BF14}) pushed this idea further, where several different $(2+1)$-dimensional interacting particle systems and random tiling models were connected through systems of non-intersecting lines. These models are believed to belong to the Anisotropic Kardar–Parisi–Zhang (AKPZ) universality class, which means that the speed function $H$ in the hydrodynamic limit \eqref{intropde} has the property that the determinant of its Hessian is negative. The domino shuffling dynamics is in particular one of them, with the connection explained by the same authors(\cite{BF15}) later in more detail. Recently, Chhita and Toninelli(\cite{CT18}) analyzed the speed and fluctuation of domino shuffling on the 2-periodic $\mbb Z^2$ lattice, and demonstrated a ``rigid'' stationary state where the fluctuation is $O(1)$.

Many of the works above focus on a specific type of region or initial condition. In terms of hydrodynamic limits starting from a more general profile, the first rigorous result in the context above was obtained in 2017 by Legras and Toninelli(\cite{FT17},\cite{LT17}). They analyzed another stochastic interface growth model from \cite{BF14}, which can be viewed as a continuous-time dynamics on lozenge tilings (the dimer model on the hexagonal lattice). In this case, different from the domino shuffling dynamics, updates at a point can depend on information arbitrarily far away, and the speed function is unbounded. As a result, the hydrodynamic limit is proved either up to the first shock time or when the initial profile is convex.

We also want to mention some background in hydrodynamic limit theory. One general approach to the hydrodynamic limit of discrete systems is to first make an educated guess about the form of the limit based on a local-equilibrium heuristic, that is, assuming the system is locally at equilibrium almost everywhere for all time. This often leads to an explicit PDE, which serves as a guide. When the PDE theory provides a characterization of a unique (weak) solution, we can try to adapt the form of the solution to the discrete system. For example, when we are dealing with the symmetric nearest neighbor simple exclusion process, as discussed in Chapter 4 of the classical reference \cite{KC13}, the hydrodynamic limit is expected to be the heat equation, whose unique solution can be characterized by an integral equation. Therefore, one wants to show that, starting from a particular initial profile,  the Riemann sum based on the empirical measure of the discrete system converges to an integral. This then becomes a classical problem in probability of showing the convergence of measures. First, one uses Prokhorov's theorem to show that every sequence of measures has a subsequential limit. Then, one uses specific knowledge about the discrete system to prove that any subsequential limit must agree with the desired integral equation, in this case using martingale techniques.

In the case when the  expected PDE is a Hamilton-Jacobi equation, the PDE theory is more complicated. When the speed function $H$ is convex or the initial profile is convex, the unique viscosity solution can be written down in a variational form, using either the Hopf-Lax formula or Hopf formula (\cite{LE14}). Certain exclusion processes do have such a PDE as the limit, and one  proof strategy consists of finding  a corresponding microscopic variational formula for the discrete system and showing the convergence of this formula to the continuous one. See, for example, the works by Sepp{\"a}l{\"a}inen(\cite{TS99}) and Rezakhanlou(\cite{FR02}). Also the Hopf formula is used in \cite{LT17} to prove the limit starting from a convex profile.

However, since  the AKPZ property exactly means that the speed function $H$ is neither convex nor concave, an explicit variational formula for the unique solution is not available. (Evans(\cite{LE14}) gave a general representation formula, but it is not clear how to relate it to these discrete systems.) The seminal work \cite{FR01} of Rezakhanlou in 2001 provided an approach in the context of certain continuous-time exclusion processes. Just as in the case of simple exclusion processes, one would like to prove the convergence of empirical measures. However, the unique viscosity solution of the Hamilton-Jacobi equation has a peculiar characterization. It involves comparing the current solution to a family of arbitrary smooth functions at all spacetime locations. Therefore, the empirical measures which we want to demonstrate convergence of need to encode the evolution from not just one initial condition, but all possible initial conditions starting from an arbitrary time, ie. as a discrete ``semigroup''. A priori, to encode this much information, the space of the resulting probability measures would be too large, i.e. inseparable, to apply Prokhorov's theorem. The key observation of Rezakhanlou is that, if the discrete system satisfies certain properties, the space of probability measures can be made separable. In \cite{FR01}, the full hydrodynamic limit of a family of exclusion processes in $d=1$ was established, with a nonexplicit Hamiltonian. It seems that \emph{our result is the first example in $d>1$ where such a full hydrodynamic limit with a nonconvex Hamiltonian has been obtained for a discrete system.} 

Another issue is that the evolution of the discrete system needs to be properly interpolated to be comparable to the evolution of the PDE, and more importantly, to keep the space of probability measures separable.  We carry out the interpolation of the domino shuffling in Section~\ref{sec:interpo}. The interpolation in \cite{FR01} is straightforward, but it takes extra work in our case due to the differences of the models. A convenience for us, however, is that the topology can be taken to be the uniform topology, instead of the Skorohod topology. This makes the argument more transparent.

In Section~\ref{sec:em}, utilizing the dimer theory, we identify the Gibbs measures of domino tilings as equilibrium measures of the shuffling process, and deduce the hydrodynamic limit starting from a flat initial condition. (Notice that we do not need the uniqueness of the dimer Gibbs measures.) This allows us to determine the full hydrodynamic limit in Section~\ref{sec:vs}. While using the general theory of viscosity solutions, we have to take care of the boundedness of the spatial gradient, imposed by our model.

The rest of the paper is outlined as follows. In Section~\ref{sec:setup}, we will provide some background information about the dimer model, and the dimer shuffling height process is defined in a general manner. We also establish a list of lemmas that are useful later. The specific set-up for the remainder of the paper and the precise statement of the theorem are presented in Section~\ref{sec:mr}. 
In Section~\ref{sec:lp}, we apply  Prokhorov's theorem and the generalized Arzel\`a-Ascoli theorem to deduce the precompactness of the sequence of the empirical measures on discrete ``semigroups'' and also prove some additional properties about the subsequential limits to be used later. In particular, the limits are bona fide semigroups.
Section~\ref{sec:comm} briefly discusses some other examples of the dimer shuffling process and possible extensions.

\section{General setup}\label{sec:setup}

\subsection{Dimer model on a periodic bipartite graph}

To start with, consider a $\mbb Z^2$-periodic bipartite graph $G=(V,E)$ embedded in the plane, where the vertices in each fundamental domain are colored  black and white in a particular way, such that the whole graph is invariant under the natural $\mbb Z^2$-translation action $T$ matching fundamental domains. One primary example will be the graph shown in Figure~\ref{fig:Z2}. Also define $G_n:=G/(n\mbb Z)^2$ as the quotient graph embedded on a torus. 

\begin{figure}[h]
\caption{The bipartite graph with vertices $\mbb Z^2$ and one fundamental domain drawn. The vertices in this figure should be considered lying on the dual lattice of Figure~\ref{fig:dominoheight}.}
\includegraphics[scale=1]{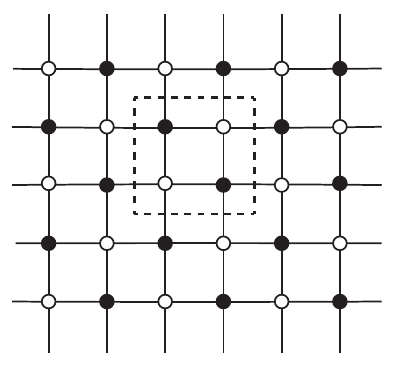}
\centering
\label{fig:Z2}
\end{figure}

A {\it dimer covering} on a bipartite graph is a subset of the edges $E$ that form a perfect matching among the vertices $V$. A chosen edge is called a {\it dimer}. We assign a nonnegative weight $w(e)$ to each edge $e$ invariant under $T$. Then on finite graphs $G_1$ and $G_n$, we can define a Boltzmann probability measure on the set of dimer coverings $\mcl M$:
\al{
\mu[M\in \mcl M]=\fr1Z\prod_{e\in M}w(e)
}
where $Z=\sum_{M\in \mcl M}\prod_{e\in M}w(e)$ is the partition function. Notice that the measure is invariant under a {\it gauge transformation}, which means multiplying all edge weights incident to a vertex by a positive constant.

This definition of course does not make sense on $G$, but we can always take a sequence of Boltzmann measures on $G_n$ with $n\rar \infty$, and any weakly convergent subsequence (which is guaranteed to exist by Prokhorov's theorem) will yield a limiting Gibbs measure on $G$, whose finite dimensional distributions are the limit along the convergent subsequence.

\subsection{Height function}\label{sec:height}

A {\it flow} on a graph is an assignment of real numbers to the directed edges,  such that edges with opposite directions are assigned opposite numbers. Given a dimer covering $M$ on $G$, we can think of it as a white-to-black flow $[M]$, where each white-to-black edge is assigned either 1 or 0. If we fix some reference covering $M_0$, then $[M]-[M_0]$ is a divergence-free flow (the net flow into each vertex is equal to 0), which induces a gradient flow on the dual graph. In other words, we can attribute a {\it height function} $h_M$ defined  on the faces to the covering $M$, by first stipulating that one base face has height 0, and assigning neighboring faces their heights as follows. When we cross an edge, the height will increase by the net amount of flow on that edge from left to right.

This height function $h_M$ is well defined on a planar graph up to an additive constant, and it is clear that, given the reference covering, we can recover the dimer covering $M$ from its height function $h_M$. If the graph is embedded on a torus, such as $G_n$ defined above, $h_M$ is still well-defined locally, but is treated as a multi-valued function globally. Suppose $h_M$ increases by $x$ as we move towards the right once around the torus back to the same face, and increases by $y$ as we move up once around the torus, we say that $h_M$ or the covering $M$ itself has {\it height change} $(x,y)$. Since $h_M$ is well defined locally, $(x,y)$ does not depend on the choice of cycles, as long as they have homology $(1,0)$ and $(0,1)$ respectively.

The set of all possible height changes on $G_1:=G/\mbb Z^2$ plays an important role in dimer theory. Their convex hull is called the {\it Newton polygon} associated to $G$, which, roughly speaking, contains exactly all possible ``slopes'' on $G$ (\cite{KOS06}). 

A different reference covering $M_0'$ will define a different height function for $M$. The difference is determined by $[M_0]-[M_0']$, which is independent of the covering $M$ of interest. Therefore, given two coverings $M$ and $M'$, the function $h_M-h_{M'}$ does not depend on the choice of reference covering.

A nice property about these height functions is that they have a lattice structure by taking pointwise maximum or minimum. This fact  was briefly mentioned in \cite{CKP01} in the context of domino tilings. Here we give a proof in the above setting.
\begin{lem}[Lattice property]\label{lem:lattice}
Fixing a reference covering $M_0$ on $G$, if $h_M$ and $h_{M'}$  are both height functions with integer values, then both $h_M\wedge h_{M'}$ and $h_M\vee h_{M'}$ are dimer height functions, where $\wedge$ denotes pointwise minimum  and $\vee$ denotes pointwise maximum. 
\end{lem}
\begin{proof}
WLOG, suppose $g=h_M\vee h_{M'}$. We first prove by contradiction that the following scenario can never happen: there is be an edge $e$ separating two faces $f_1$ and $f_2$ such that $g(f_1)=h_M(f_1)>h_{M'}(f_1)$ and $g(f_2)=h_{M'}(f_2)>h_M(f_2)$. Let us assume that when we cross $e$ from $f_1$ to $f_2$, the white vertex is on the left. If $e\in M_0$, a dimer height function either stays the same or decreases by 1 going from $f_1$ to $f_2$. Then we must have $h_M(f_2)\geq h_M(f_1)-1\geq h_{M'}(f_1)\geq h_{M'}(f_2)$, a contradiction. Similarly, if $e\notin M_0$, a dimer height function either stays the same or increases by 1 going from $f_1$ to $f_2$. Then $h_{M'}(f_1)\geq h_{M'}(f_2)-1\geq h_M(f_2)\geq h_M(f_1)$, again a contradiction.

Now suppose $g$ is not a dimer height function. One possible obstacle is that for some edge $e$ separating two faces $f_1$ and $f_2$,  $g(f_1)-g(f_2)$ is not one of the allowed dimer height changes across $e$. Then the scenario above must happen, which is impossible.

Another possible obstacle is when the value of $g$ changes at least four times circling around a vertex $v$. (Obviously $g$ has to change even number of times circling around any vertex.) On the other hand, any dimer height function must change either zero or two times circling around $v$. Also if $e_0$ is the unique edge in $M_0$ that connects $v$, a dimer height function that changes two times circling around $v$ must change its value across $e_0$. This implies that out of the four value changes for $g$ around $v$, at least one change corresponds to the scenario mentioned before, which is again impossible.

The last possible obstacle is when the value of $g$ changes exactly two times circling around a vertex $v$, but neither of them happens at $e_0$, borrowing the notations from above. Since the scenario mentioned at the beginning is impossible, those two changes must have one of them coinciding with $h_M$ and the other coinciding with $h_{M'}$. Therefore both $h_M$ and $h_{M'}$ must both change across $e_0$, with the same net increase. But since $g =  h_M\vee h_{M'}$, $g$ must also change across $e_0$, a contradiction.
\end{proof}

\subsection{Local moves}

\label{sec:localmove}

Now we define two types of local moves for the dimer model,  {\it vertex contraction/expansion} and the {\it spider move}. These first appeared in \cite{JP03} and were also studied in \cite{AK11,DS07}. These local moves happen at two different levels. One level is a local modification of the weighted bipartite graph, and another level is a possibly random mapping of dimer coverings on the graph.

{\it Vertex contraction/expansion}, on the graph level, involves either shrinking a $2$-valent vertex and its two incident edges into a single vertex, or its reversal. See Figure~\ref{fig:shrink}. Assume that {\it the three vertices involved are all distinct}. Before shrinking, we first perform a gauge transformation to make both edge weights equal to 1. During expansion, on the other hand, simply assign both edges weight 1. 
\begin{figure}[h]
\caption{Vertex contraction/expansion move}
\includegraphics[scale=1.3]{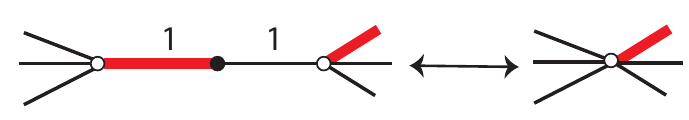}
\centering
\label{fig:shrink}
\end{figure}

On the dimer level, given a dimer covering before contraction, we simply delete the dimer incident to the deleted middle vertex, and keep the rest of the dimer covering after contraction. For expansion, we keep the covering, and match the added middle vertex with the unmatched side.

The spider move is the more interesting case. See Figure~\ref{fig:spider} for an illustration. 
\begin{figure}[h]
\caption{Spider move on graph level}
\includegraphics[scale=1.3]{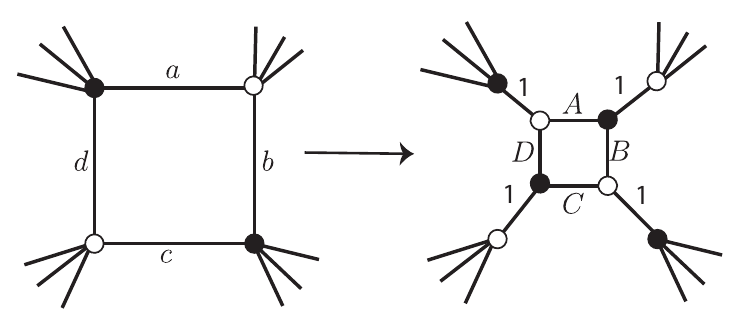}
\centering
\label{fig:spider}
\end{figure}

Start with a quadrilateral face with a top-left black vertex such that {\it all four vertices are distinct}. On the graph level, first insert four tentacles at the corners. The four tentacles are  assigned weight 1. The other four weights are assigned as follows:
\al{
A=\fr{c}{ac+bd}, B=\fr{d}{ac+bd}, C=\fr{a}{ac+bd}, D=\fr{b}{ac+bd}.\numberthis\label{spider}
} 
For the spider move with the opposite coloring, first perform vertex expansion at four corners and then implement the spider move on the internal face. For the reversal with the original coloring, perform the opposite-coloring spider move on the internal face and shrink the four $2$-valent edges.  Hence we  only call the move in Figure~\ref{fig:spider} the spider move.

On the dimer level, we keep all the dimers that are not one of the four internal edges. Then depending on whether each of four original vertices were matched externally or internally, we make some choices in order to complete a dimer covering. See Figure~\ref{fig:spidershuffle} for some of the cases.
\begin{figure}[h]
\caption{Spider moves on the dimer level. Only the first row has randomness. In the second row, we omitted three other symmetric cases.}
\includegraphics[scale=0.9]{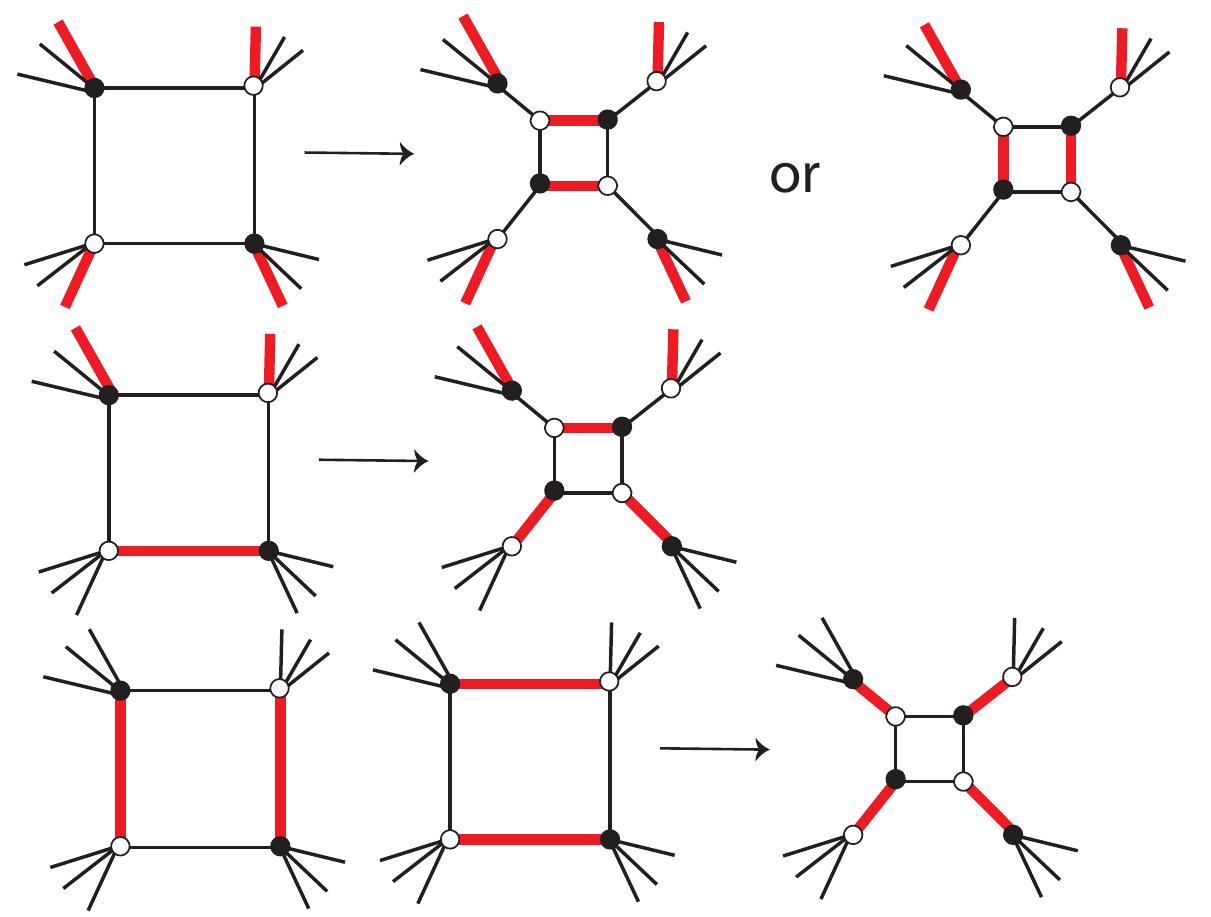}
\centering
\label{fig:spidershuffle}
\end{figure}

In the first row of Figure~\ref{fig:spidershuffle}, a choice is made between the two possibilities according to the ratio between their total weights. The two horizontal dimers are chosen with probability $\fr{AC}{AC+BD}$ and the two vertical dimers are chosen with probability $\fr{BD}{AC+BD}$. Many variants of the following statement appeared in the literature(\cite{AK11,JP03, DS07}), and we will include a proof at the end for completeness.

\begin{prop}\label{prop:loc}
For any finite bipartite  graph $H$ embedded on a torus, the local moves preserve the Boltzmann measure of dimer coverings on $H$, in the sense that applying a local move to the Boltzmann measure on $H$ results in the Boltzmann measure on the new graph $H'$.
\end{prop}

\subsection{Shuffling height process}\label{sec:dimershuffle}

The following is not a precise definition, but rather a general description of the type of process we are considering.

First, we consider a {\it global operation} where we choose a local move on $G$ or $G_n$, and perform it at all $T$-periodic counterparts simultaneously, requiring that {\it the edges involved do not overlap with each other}. The spider moves at different locations are independent in terms of their randomness. See Figure~\ref{fig:dimershuffle} for an example. 

By Proposition~\ref{prop:loc}, such a global operation still preserves the Boltzmann measure, since we can consider it as performing the local moves sequentially. It is also well defined on $G$, as we can perform the local moves in the increasing order of  their distance from the origin, so that every finite region of $G$ will be determined after some finite number of local moves.

Second, we want to compare the height functions before and after a local move. Given the reference covering $M_0$ and a height function $h_M$ before a local move, we may choose (in most cases there is a natural choice) the reference covering after the local move to be deterministically one of the possible outcomes of the local move applied to $M_0$. This defines a new height function $h_{M'}$ for the random outcome $M'$. 

\begin{lem}\label{lem:locheight}
With the choice above of the reference covering on $G$, $h_{M'}$ agrees with $h_M$ at every face except the inner face of the spider move (up to a global additive constant which is made zero), where $M'$ is an outcome of $M$ after a local move. 
\end{lem}

For a moment let us forget about the precise embedding. When we say $h_{M'}$ agrees with $h_M$ at a face, we mean that the heights defined at the {\it combinatorially corresponding faces} agree, since no face is created or destroyed.

\begin{proof}
We can relate the heights at different faces as in Figure~\ref{fig:localheight}. Since we assume that all vertices involved in the local moves are distinct, the dimer configuration along the dotted paths drawn remains the same. Therefore, the height on the surrounding faces can be made invariant before and after the local move. Then the height at every other face also stays unchanged except the middle one in the spider move.
\begin{figure}[h]
\caption{Paths in dual graphs}
\includegraphics[scale=1]{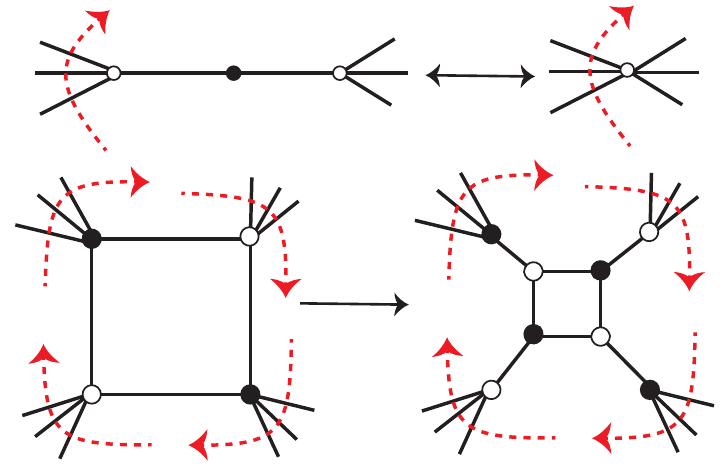}
\centering
\label{fig:localheight}
\end{figure}
\end{proof}

A consequence of Lemma~\ref{lem:locheight} is that during a global operation, the heights of all the faces except the inner faces of spider moves can be kept unchanged, since such is true after every local move. This is the reason why we will be able to discuss the evolution of height functions. Also in this case, we will choose the same new reference dimer for each local move. That is, we can assume that {\it the reference covering remains $T$-periodic} after a global operation. 
\iffalse

The same argument also leads to the following lemma.
\begin{lem}\label{lem:heightconst}

The height change $(x,y)$ of a dimer covering on a torus remains the same after a local move. 
\end{lem}
\begin{proof}
Since the definition of the height change $(x,y)$ only depends on the homology of the measuring cycle, we can slide it locally to avoid the edges involved in the local move, by utilizing the dashed paths drawn in Figure~\ref{fig:localheight}.
\end{proof}
\fi

When there are two dimer coverings $M_1$ and $M_2$ on $G$, we can perform a local move or a global operation on them simultaneously. The only requirement is that,  they are {\it coupled} so that each pair of corresponding faces during a spider move share the same randomness.  Then we have the following ``monotonicity'' lemma.

\begin{lem}[Monotonicity]\label{lem:monotone}
With the same choice of reference covering on $G$ as in Lemma~\ref{lem:locheight}, if $h_{M_1}\geq h_{M_2}$ at each corresponding face before a local move, then $h_{M_1'}\geq h_{M_2'}$ still holds afterwards, where $M_1'$ and $M_2'$ are the coupled results of the local move performed on $M_1$ and $M_2$ respectively.
\end{lem}
\begin{proof}
The heights of the faces do not change in vertex contraction/expansion, so there is nothing to prove there.

It suffices to consider $h_{M_1} -  h_{M_2}$ and $h_{M_1'} - h_{M_2'}$, since they do not depend on the choice of reference covering. By assumption $h_{M_1}-h_{M_2}\geq 0$. By Lemma~ \ref{lem:locheight}, we can assume that, for both $h_{M_1}$ and $h_{M_2}$, the heights at all other faces stay the same except the internal face of the spider move, denoted as $f$. 

Therefore, we have $h_{M_1'}- h_{M_2'}\geq 0$ at all faces except $f$. Suppose the statement is wrong, then $h_{M_1'}-h_{M_2'}<0$ at $f$. By considering the flow $[M_1']-[M_2']$, the only case this can happen is shown in Figure~\ref{fig:monotonepf}, where the green vertical dimers are in $M_2'$ and the red horizontal dimers  are in $M_1'$. This cannot happen since we assumed that they are coupled at $f$ during the spider move.
\begin{figure}[h]
\caption{A local configuration of two dimer coverings}
\includegraphics[scale=1]{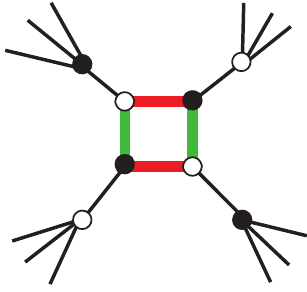}
\centering
\label{fig:monotonepf}
\end{figure}
\end{proof}

Now to obtain a tractable height process, we make a strong assumption. 

{\bf Assumption:}  After a particular choice of sequence of global operations being performed on $G_n$ (resp. $G$), the resulting graph coincides with $G_n$ (resp. $G$) in  terms of the actual embedding,  with combinatorially corresponding faces matching each other, and  the edge weights are also the same up to gauge transformations.

Another assumption is that the reference dimer covering also remains the same. This is not necessary because there are only finitely many $T$-periodic coverings, so we can make this true by running longer time if necessary, given the assumption above.

If the assumption above is true, we call one iteration of such a sequence of global operations a {\it shuffle}, and the corresponding height evolution a {\it shuffling height process}.

The assumption might seem very strong. The specific example that we will analyze is  the one in Figure~\ref{fig:Z2}. But already in $\mbb Z^2$ lattice, by increasing the size of the fundamental domain, some special phenomena in the steady state fluctuation are discovered in \cite{CT18}. 

The reason why we introduced this process in this more general manner is, first, the existence of the hydrodynamic limit of any such process can be obtained by a similar approach, even though the specific PDE might be hard  to compute; and second, the necessary lemmas listed above and their proofs do not assume much about the specific graph structure, so it seems more natural to state them independently.

\section{The main result}\label{sec:mr}

\subsection{The specific example}\label{sec:example}

Now we turn to the simplest example where a shuffling height process can be defined, the 1-periodic domino tiling model in Figure~\ref{fig:Z2}.

We assume that the vertical edges have weight $\sqrt{a}$ and the horizontal edges have weight 1, as any other choice of positive weights with the same fundamental domain is gauge equivalent to this one. By convention in the literature, we choose the reference flow $[M_0]$ in the initial graph to be $1/4$ on each edge.  We define the  height function $h$ on the faces, which are labeled by  coordinates in $\mbb Z^2$. With this coordinate, the graph is periodic under the action $T$, generated by translations $(2,0)$ and $(0,2)$. 

We assume that initially the face at $(0,0)$ has a top-left black vertex. We call faces $(i,j)$ with $i+j$ even {\it even faces}, and otherwise {\it odd faces}. Also, we multiply the heights by $4$, so that all of them are integers. This is the height function of domino tilings defined in \cite{WT90}. Figure~\ref{fig:dominoheight} shows a picture in terms of dominoes.
\begin{figure}[h]
\caption{A domino tiling on a so-called Aztec diamond region with heights labeled. The vertices lie on the dual lattice of Figure~\ref{fig:Z2}.}
\includegraphics[scale=1.2]{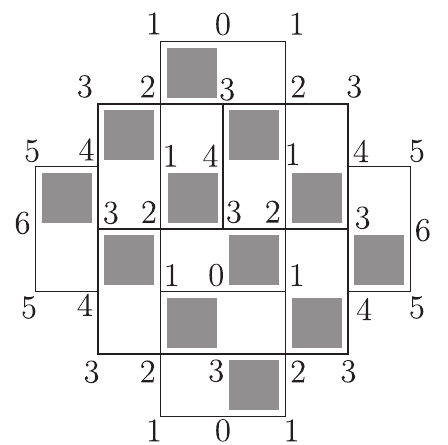}
\centering
\label{fig:dominoheight}
\end{figure}

To be clear, when we speak of an ``edge'', we always refer to an edge on the primal graph as in Figure~\ref{fig:Z2}, etc. By definition of the height function, when we cross an edge with a white vertex on the left, the height either increases by 3 or decreases by 1. Therefore once we fix the height of face $(0,0)$ modulo 4, all other faces are determined modulo 4.

The shuffling procedure is just the domino shuffling from \cite{EKLP92}. Propp \cite{JP03} rephrased this shuffling procedure in terms of the local moves described in Section \ref{sec:localmove}. By our definition, a shuffle consists of the following steps.

\begin{enumerate}
\item Perform a spider move at all even faces $(i,j)$ (these are two $T$-periodic families of local moves);

\item Perform vertex contraction at all 2-valent vertices;

\item Perform a spider move at all odd faces $(i,j)$;

\item Perform vertex contraction at all 2-valent vertices;

\end{enumerate}

See Figure~\ref{fig:dimershuffle} for one iteration. By  formula~\eqref{spider}, after Step 1, the horizontal edge weights become $\fr1{1+a}$, and the vertical edge weights become $\fr{\sqrt{a}}{1+a}$. So up to gauge transformations, the edge weights remain the same after Step 2, and after Step 4 as well. By viewing the reference flow $[M_0]$ as a convex combination of four different integer coverings, each consisting of a single type of edges, $[M_0]$ also remains $1/4$ on all horizontal and vertical edges. Therefore, the assumption for a shuffling height process is satisfied. 
\begin{figure}[h]
\caption{One domino shuffle with a fixed fundamental domain labeled. Each step consists of several $T$-periodic families of local moves.}
\includegraphics[scale=1.7]{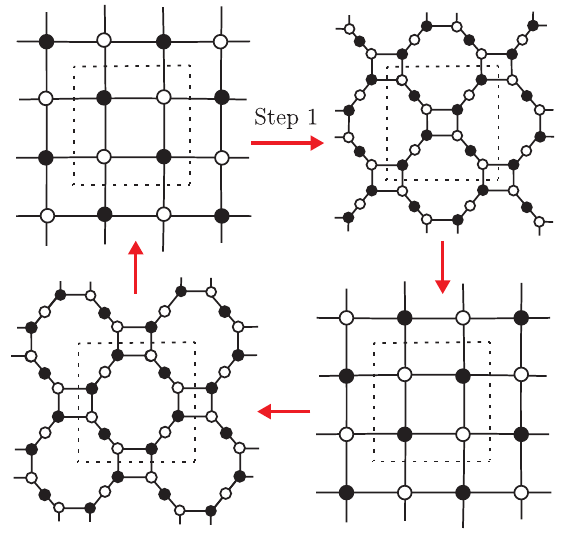}
\centering
\label{fig:dimershuffle}
\end{figure}

Consider the height at certain even face $(i,j)$. By Lemma~\ref{lem:locheight}, it only gets modified at Step 1 of a shuffle, since that is the only time when that face undergoes a spider move. A small inconvenience is that the height module $4$ changes after a shuffle. To compensate for that, from now on, {\it we subtract all heights by $2$ after every shuffle}. This amounts to adding a constant drift to the hydrodynamic limit, which does no harm. Since the only relevant faces during a spider move are the neighboring ones, we can list all the possible outcomes at $(i,j)$ after a shuffle in Table~\ref{tab:1}. We see that the height at $(i,j)$ is nonincreasing, and stays the same modulo $4$. The same table also describes the height evolution at an odd face $(i,j)$, where the left column represents the height after Step 2. In other words, the entire height process can be defined using the local rules described in Table~\ref{tab:1}, without mentioning dimers.
\begin{table}[h]
\centering
\caption{The height evolution at an even face $(i,j)$ during a shuffle }
\label{tab:1}
\begin{tabular}{@{}rclc@{}}\toprule[3 pt]
\multicolumn{3}{c}{Local heights centered at $(i,j)$} & The height at $(i,j)$ after the shuffle\\\midrule
&$h+1$&&\\
$h-1$&$h$&$h-1$&$h-4$\\
&$h-3$&&\\\cmidrule(r{4em}){1-3}
&$h-3$&&\\
$h-1$&$h$&$h-1$&$h-4$\\
&$h+1$&&\\\cmidrule(r{4em}){1-3}
&$h+1$&&$h$ with probability $\fr{a}{1+a}$ \\
$h-1$&$h$&$h-1$& \\
&$h+1$&&$h-4$ with probability $\fr{1}{1+a}$\\\cmidrule(r{4em}){1-3}
&$h+1$&&\\
$h-1$&$h$&$h+3$&$h$\\
&$h+1$&&\\\cmidrule(r{4em}){1-3}
&$h+1$&&\\
$h+3$&$h$&$h-1$&$h$\\
&$h+1$&&\\\cmidrule(r{4em}){1-3}
&$h-3$&&\\
$h-1$&$h$&$h-1$&$h-4$\\
&$h-3$&&\\\cmidrule(r{4em}){1-3}
&$h+1$&&\\
$h+3$&$h$&$h+3$&$h$\\
&$h+1$&&\\
\bottomrule[3 pt]
\end{tabular}
\end{table}

\subsection{Hydrodynamic limit}\label{sec:hydro}

To state a hydrodynamic limit result on the height evolution, we need to introduce a time parameter $t$. Suppose the initial condition at $t=0$ is the height function of certain dimer covering on $G$, which is a function  $\mbb Z^2 \rar \mbb Z$, as defined previously, the dynamics is that a shuffle happens at $t=1,2,3,\dots$. This defines a random process $h:\mbb Z^2\times \mbb N\rar \mbb Z$ where $h(x,t)$ denotes the height at face $x$ and time $t$. A priori, the spatial function $h(\cdot,t)$ for any $t$ is a domino height function. As mentioned before, $h(x,\cdot)$ is nonincreasing. Furthermore, we require that $h(0,0)\equiv 0 \pmod{4}$. This immediately determines all $h(x,t) \pmod{4}$. We will call such height function $h(\cdot, t)$ {\it admissible}.

 The underlying probability space $\Omega$ consists of a collection of iid Bernoulli random variables at each $(x,t)\in\mbb Z^2\times \mbb N$, which take value $1$ with probability $\fr 1{1+a}$ and value $0$ with probability $\fr a{1+a}$. In particular, given $\omega\in \Omega$, $\omega(x,t)$ dictates the randomness of a spider move that happens at face $x$ and time $t$.

Define the space of {\it asymptotic height functions} $\Gamma$ to be the set of all {\it $2$-spatially-Lipschitz functions} from $\mbb R^2$ to $\mbb R$, which in this paper means
\al{
|f(x)-f(y)|\leq 2\left|x-y\right|_\infty,
}
where $|\cdot|_p$ is the $\ell_p$ norm.

The choice of $\Gamma$ comes from the Newton polygon, as defined in Section~\ref{sec:height}. In this case, the (rescaled) Newton polygon bounds the region $U:=\lc x:|x|_1\leq 2\rc$, and a differentiable function is in $\Gamma$ iff its gradient lies in $U$. See Lemma~\ref{lem:phider} for a similar statement.

Now suppose we are given some $g\in \Gamma$ and a sequence of initial conditions $(h_n(\cdot,0))_{n\in \mbb N}$ approximating $g$. What we mean exactly is that $(h_n(\cdot,0))$ is a sequence of admissible height functions, random or not, independent or not, such that
\ba{
\lim_{n\rar \infty} \mbb{E} \sup_{|x|_1\leq R}\left|\fr1nh_n(\lfloor nx\rfloor, 0)-g(x)\right|=0\label{init}
}
for every finite $R>0$, and the expectation $\mbb{E}$ is taken over the probability space $\Omega_0$ of the initial condition $(h_n(\cdot,0))_{n\in \mbb Z_{>0}}$.

The height evolution of $(h_n)$ is governed by the single probability space $\Omega$. This means implicitly that the randomness of a spider move at $(x,t)$ is coupled for all $h_n$.

Now we are ready to state the hydrodynamic limit.
\begin{thm}\label{thm:main}We have for every $R>0$
\ba{\lim_{n\rar \infty} \mbb{E}\sup_{|x|_1\leq R, t\leq R} \left|\fr1nh_n(\lfloor nx\rfloor, \lfloor nt\rfloor)-u(x,t)\right|=0\label{hydro}}
where $u:\mbb R^2\times \mbb R_{\geq 0}\rar \mbb R$ is the unique viscosity solution of 
\begin{equation}
\lc\begin{array}{rl}u_t+H(u_x)&=0\label{pde}\\
u(x,0)&=g(x),\end{array}\right.
\end{equation}
 $H$ is defined on $U$ by
\ba{H(\rho_1,\rho_2)=\fr4{\pi}\cos^{-1}\lp \fr{a}{1+a}\cos\lp\fr{\pi \rho_2}2\rp - \fr{1}{1+a}\cos\lp\fr{\pi \rho_1}2\rp\rp\label{hamil},}
 $u_t:=\pd{u}{t}$ and $u_x$ denotes the spatial gradient of $u$.
\end{thm}
The precise definition of viscosity solutions is delayed to  Section~\ref{sec:vs}. 

\begin{rem}
Notice that $H$ is continuous on $U$. One can compute the determinant of the Hessian of $H$ to be
\al{
-\lp\fr{a\pi\lp\cos\lp \fr{\pi \rho_1}2\rp+\cos\lp\fr{\pi \rho_2}2\rp\rp}{(a+1)^2-\lp\cos\lp\fr{\pi\rho_1}2\rp-a\cos\lp\fr{\pi \rho_2}2\rp\rp^2}\rp^2
}
which is negative in the interior of $U$.

The result here is for the shuffling in the plane. It can be replaced by a torus or cylinder, and the formula remains the same.
\end{rem}

\section{Smoothing out the height process}\label{sec:interpo}

The goal of this section is to embed the height process $h(s,t)$ in a suitable  space, which, as shown later,  also contains the semigroup solving the PDE. Since $h$ is discrete in both space and time, we need to extend it to a continuous process.

\subsection{Useful properties of the height process}

We first take a closer look at our height process $h(x,t)$. Let $\Phi$ denote the space of all admissible height functions, which are domino height functions whose value at $(0,0)$ is $0\pmod 4$. The following lemma is  rephrasing Lemma~\ref{lem:lattice}.
\begin{lem}[Lattice property]\label{lem:lattice2}
If $\vph_1, \vph_2\in \Phi$, then $\vph_1\wedge \vph_2\in \Phi$ and $\vph_1\vee \vph_2\in \Phi$. 
\end{lem}

Define $h(x,t;\vph,\omega)$ as the height at position $x$ and time $t$ of the deterministic height process with an initial configuration $h(\cdot, 0)=\vph\in \Phi$ and Bernoulli mark $\omega\in \Omega$.
 
Since $\vph\in \Phi$ is well defined up to a global additive constant that is a multiple of $4$, for all $k\in \mbb Z$,
\ba{
h(x,t;\vph+4k, \omega)=h(x,t;\vph, \omega) +4k.\label{moveup}
}

The following ``monotonicity'' lemma is the deterministic version of Lemma~\ref{lem:monotone}.
\begin{lem}[Monotonicity]\label{lem:monotone2}
Given $\vph_1, \vph_2\in \Phi$, if $\vph_1\leq \vph_2$, then $h(\cdot, t;\vph_1,\omega)\leq h(\cdot, t;\vph_2,\omega)$ for all $t\in \mbb N$.
\end{lem}

Now we state a simple but crucial  lemma, which states that  information propagates at linear speed. This can also be easily generalized to other shuffling height processes. 
\begin{lem}[Linear propagation]\label{lem:linear}
Given $\vph_1, \vph_2\in \Phi$ and $x\in \mbb Z^2$, if $\vph_1(y)=\vph_2(y)$ for all $y$ such that $|x-y|_1 \leq R$, then $h(y, t;\vph_1,\omega)= h(y, t;\vph_2,\omega)$ for all $y$ such that $|x-y|_1\leq R-2t$.
\end{lem}
\begin{proof}
During each shuffle, there are two rounds of height updates. In the first step, all even faces $(i,j)$ get modified. But since $\omega$ is given, the new height at $(i,j)$ is just a function of the heights of its four neighbors. Similarly, in the third step, heights at odd faces $(i,j)$ are updated according to its four neighbors. Therefore, the new height at any face $x$ after a shuffle is just a function of original heights at $y$ where $|y-x|_1\leq 2$. Now the statement follows by an induction on $t$.
\end{proof}

Combing the few statements, we deduce a ``localization'' property of shuffling height processes.
\begin{prop}[Localization]\label{prop:local}
Assume $k \in \mbb N$.
~\begin{enumerate}
\item{ Given $\vph_1, \vph_2\in \Phi$ and $x\in\mbb Z^2$, if $|\vph_1(y)-\vph_2(y)|\leq 4k$ for all $y$ such that $|x-y|_1\leq R$, then $|h(y, t;\vph_1,\omega)- h(y, t;\vph_2,\omega)|\leq 4k$ for all $y$ such that $|x-y|_1\leq R-2t$;}
\item{ Given $\vph_1, \vph_2\in \Phi$, if $|\vph_1(y)-\vph_2(y)|\leq 4k$ for all $y$, then $|h(y, t;\vph_1,\omega)- h(y, t;\vph_2,\omega)|\leq 4k$ for all $y$.}
\end{enumerate} 
\end{prop}
\begin{proof}
Let $\vph_3=\vph_1\vee(\vph_2+4k)$, which is admissible by Lemma~\ref{lem:lattice2}. Since $\vph_1(y)\leq \vph_2(y)+4k$ for all $y$ such that $|x-y|_1\leq R$, we have $\vph_3(y)=\vph_2(y)+4k$ for all such $y$. By Lemma~\ref{lem:linear}, $h(y, t;\vph_3,\omega)= h(y, t;\vph_2+4k,\omega)$ for all $y$ such that $|x-y|_1\leq R-2t$. Therefore, for all such $y$,
\al{
h(y,t;\vph_1,\omega)&\leq h(y,t;\vph_3,\omega)\\
&=h(y, t;\vph_2+4k,\omega)\\
&=h(y, t;\vph_2,\omega)+4k.\;\;\;\;\;\;\;\text{(by \eqref{moveup})}
}

The other inequality can be proved similarly, so the first statement holds. The second statement follows by taking $R=\infty$.
\end{proof}

Another property of the height process is that the vertical drift speed is linearly bounded.
\begin{lem}[Vertical speed bound]\label{lem:vert}
For any $\vph\in \Phi$, $x\in\mbb Z^2$ and $t\in\mbb N$, 
\al{\vph(x)-4t\leq h(x,t;\vph,\omega)\leq \vph(x).}
\end{lem}
\begin{proof}
This follows easily from Table~\ref{tab:1} and the same table at odd faces $(i,j)$.
\end{proof}

 Define the space translation operator $\tau_y$ for $y\in \mbb Z^2$ on both height functions $h$ and $\omega\in \Omega$ by
\al{
\tau_yh(x)=h(x-y), \tau_y\omega(x,t)=\omega(x-y,t)
}
for every $x\in \mbb Z^2$ and $t\in \mbb Z$. Then we have 
\ba{
h(x-2y,t;\vph,\omega)=h(x,t;\tau_{2y}\vph, \tau_{2y}\omega)
}
for every $x,y\in \mbb Z^2$, $\vph\in \Phi$, $\omega\in \Omega$. The factor $2$ is present so that $\tau_{2y}\vph\in \Phi$.

Another observation is that $h$ satisfies a semigroup-like property. Define the time translation operator $\gamma_s$ for $s\in \mbb N$ on $\omega$, by $\gamma_s\omega(x,t)=\omega(x,t+s)$. Then for all $s,t\in \mbb N$, $s\leq t$,
\ba{
h(x,t;\vph,\omega)=h(x,t-s;h(\cdot,s;\vph,\omega),\gamma_s\omega).\label{semi}
}

\subsection{Smoothing out the height process spatially}\label{sec:smoothspace}

Define the pyramid height function $v:\mbb Z^2\rar \mbb Z$ to be
\al{
v\lp x\rp=\min\lc \vph(x):\vph\in \Phi, \vph(0)=0\rc.
}
Here we can take pointwise minimum because the height difference between 0 and $x$ is a bounded integer. This is an admissible height function due to Lemma~\ref{lem:lattice2}. To help visualize, the corresponding dimer covering is shown in Figure~\ref{fig:pyramid}.
\begin{figure}[h]
\caption{The pyramid height function $v$}
\includegraphics[scale=1.8]{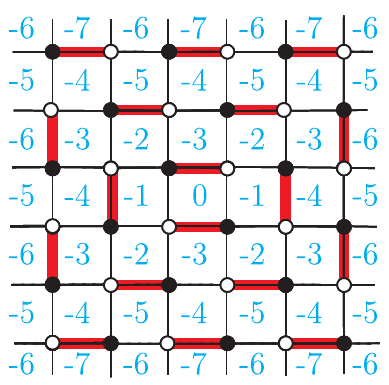}
\centering
\label{fig:pyramid}
\end{figure}

To convince ourselves that this is exactly $v$, notice that from the origin to any other face, there exists a face path such that the height decreases by 1 every step. Since the height can either decrease by 1 or increase by 3 at those steps, we have the correct height at every face. Observe that $v(x)=v(-x)$.

A more constructive way to describe $v$ that also works for  different graphs $G$ is the following. Consider the boundary vertices of the Newton polygon $\lc x: |x|_1= 2\rc$. Each of them corresponds to a covering on $G_1$. Lift them to periodic coverings on $G$, and find their height functions that equal to $0$ at the origin. Now take the pointwise minimum.

Fix $n\in \mbb Z_{>0}$. Let us start with an input function $g\in \Gamma$. We will define a height function $\vph_g^n\in \Phi$ close to $g$, and use it as the initial condition for the height process. Set
\al{
\Phi_g&:=\lc \vph\in \Phi, \vph=\tau_{y}v +k  \text{ for some }y\in \mbb Z^2, k\in \mbb Z \text{ such that }k\leq ng\lp\fr{y}{n}\rp\rc,\\
\vph_g^n(x)&:=\max_{\vph\in \Phi_g}\lc\vph(x)\rc.
}
To see that $\vph^n_g$ is well defined,  take any $x, y\in \mbb Z^2$, and $k\in \mbb Z$ such that $k\leq ng\lp \fr yn\rp$. Since $g$ is 2-spatially-Lipschitz, 
\ba{
\left|ng\lp\fr y n\rp-ng\lp\fr x n\rp \right|&\leq 2n\left|\fr y n-\fr x n\right|_\infty=2\left| y - x \right|_\infty,\nonumber\\
\Rightarrow ng\lp\fr y n\rp&\leq ng\lp\fr x n\rp+ 2\left| y - x \right|_\infty.\label{2lip}
}

Then observe that $v(x)\leq -2|x|_\infty+1$, so we have
\ba{(\tau_yv)(x)+k&\leq -2|x-y|_\infty+1+ng\lp\fr yn\rp\nonumber\\
&\leq -2|x-y|_\infty+1+ng\lp\fr x n\rp+ 2\left| y - x \right|_\infty=1+ng\lp\fr x n\rp\label{vphupb}
}
where we used \eqref{2lip} for the second inequality. 

On the other hand, there must exist $k\in \mbb Z$ such that $ng\lp \fr xn\rp-4 < k\leq ng\lp \fr xn\rp$ and $\tau_x v+k\in \Phi$, and with such $k$,
\ba{
(\tau_x v)(x)+k=k>ng\lp \fr xn\rp-4. \label{vphlob}
}

Combining \eqref{vphupb} and \eqref{vphlob}, we conclude that 
\ba{ng\lp \fr xn\rp-4<\vph^n_g(x)\leq ng\lp\fr x n\rp+1. \label{g2hbd}}

Due to Lemma~\ref{lem:lattice2}, $\vph^n_g$ is an admissible height function. Furthermore, we can show that $\vph^n_g$ is determined locally by $g$ at every $x$. This property  will be important in Section~\ref{sec:smoothtime}.
\begin{prop}\label{prop:locdef}
Given $g\in \Gamma$, $\forall x\in \mbb Z^2$,
\al{
\vph_g^n(x)=\vph(x)
}
for some $\vph\in \Phi$ such that $\vph=\tau_{y}v+k$, $\left|y-x\right|_1\leq 1$, $k\leq ng\lp \fr{y}n\rp$.
\end{prop}

\begin{proof}
Consider some $x\in \mbb Z^2$. By definition, there exists $y\in \mbb Z^2$, $k\in \mbb Z$ such that $\vph_1=\tau_yv+k\in \Phi_g$, $
\vph_g^n(x)=\vph_1(x)=v(x-y)+k$, and 
\ba{
k\leq ng\lp\fr{y}n\rp.\label{446}
}
Assume $|y-x|_1\geq 2$, otherwise there is nothing to prove.

Again, since $g$ is 2-spatially-Lipschitz, 
\ba{
ng\lp\fr x n\rp&\geq ng\lp\fr y n\rp-2\left| y - x \right|_\infty.\label{448}
}
If $v(x-y)= -2\left| y - x \right|_\infty$, then we claim that $\vph_2=\tau_x v+k+v(x-y)\in \Phi_g$ and $\vph_2(x)=\vph_1(x)$.

Indeed, we have $\vph_2(x)=v(x-x)+k+v(x-y)=k+v(x-y)$. Since $\vph_1\in \Phi$ and $\vph_1(x)=v(x-y)+k$, in particular $\vph_2$ is also in $\Phi$. Furthermore,
\ba{k+v(x-y)=k-2\left| y - x \right|_\infty\leq ng\lp\fr{y}n\rp-2\left| y - x \right|_\infty\leq ng\lp \fr xn\rp,}
by \eqref{446} and \eqref{448}, so the claim is true. The statement then follows with candidate $\vph_2$.

We are left with the case when $v(x-y)\neq -2\left| y - x \right|_\infty$. Observe that $v(x-y)= -2\left| y - x \right|_\infty$ iff $y-x=(i,j)$ with $i+j$ even. Therefore, if $v(x-y)= -2\left| y - x \right|_\infty$ does not hold for some $x$, there is a neighboring face $x'$ of $x$, such that $|x'-y|_\infty=|x-y|_\infty-1$, and $v(x'-y)= -2\left| y - x' \right|_\infty$. Let $\vph_3=\tau_{x'} v+k+v(x'-y)$. Since $\vph_3(x')=v(x'-x')+k+v(x'-y)=k+v(x'-y)=\vph_1(x')$, by the same argument as above, $\vph_3\in \Phi_g$.

We have $\vph_3(x)=\tau_{x'}v(x)+k+v(x'-y)=v(x-x')+k+v(x'-y)$. With the help of  Figure~\ref{fig:pyramid}, it is easy to verify that $v(x-x')+v(x'-y)=v(x-y)$, knowing that $|y-x|_1\geq 2$, $y-x=(i,j)$ with $i+j$ odd, $|x'-x|_1=1$,  and $|x'-y|_\infty=|x-y|_{\infty}-1$. Therefore, the statement holds with candidate $\vph_3$.
\end{proof}

With some fixed $s,t\in\mbb N$, $s< t$, $g\in \Gamma$, $\omega\in \Omega$, consider the height process
\al{
h\lp  x, t-s; \vph^n_g, \gamma_{s}\omega\rp
}
as a function of $x\in \mbb Z^2$ only. Its direct linear interpolation is not in $\Gamma$, because when $x$ changes by $1$, the function might change by $3$. Instead, we define a new function $\psi_{s,t}$ such that for all $x\in \mbb Z^2$,
\al{
 \psi_{s,t}(2x):= h\lp  2x, t-s; \vph^n_g, \gamma_{s}\omega\rp.
}
The function $\psi_{s,t}$, for now, is only defined on $2\mbb Z^2$, where $k\mbb Z^2$ for some constant $k$ denotes the set $\lc (ki, kj):(i,j)\in\mbb Z^2\rc$. In other words, $\psi_{s,t}$ agrees with the height process at all the $T$-translations of the origin. Then $\psi_{s,t}$ is 2-spatially-Lipschitz on $2\mbb Z^2$, by checking the pyramid height function. 

We want to further interpolate $\psi_{s,t}$ to a 2-spatially-Lipschitz function $\mbb R^2\rar \mbb R$. More specifically, given the heights at the four faces, listed in counterclockwise order,
\al{
(2i,2j), (2i+2, 2j),(2i+2,2j+2), (2i, 2j+2),
}
we first interpolate $\psi_{s,t}$ along the four sides linearly. Inside the square $[2i,2i+2]\times [2j,2j+2]$, we need to be careful about interpolating $\psi_{s,t}$ to keep it 2-spatially-Lipschitz. See Appendix~\ref{appenb} for an explicit interpolation.

Due to the finiteness of the fundamental domain,
\ba{\left|\psi_{s,t}(x)-h\lp x, t-s; \vph^n_g, \gamma_{s}\omega\rp\right|<{C_0}, \forall x\in \mbb Z^2 \label{psidif}} 
for some global constant $C_0>0$. We leave it as $C_0$ because this depends on the interpolation.

Now given $s,t\in \fr1n \mbb N$, $s<t$, we define
\al{
S_n(s,t;g,\omega)(x):=\fr1n\psi_{ns,nt}(nx).
}
Since $\psi_{ns, st}$ is in $\Gamma$, $S_n(s,t;g,\omega)$ is in $\Gamma$ as well.
From \eqref{psidif}, we deduce that
\ba{\left|S_n(s,t;g,\omega)(x)-\fr1n h\lp n x, n(t-s); \vph^n_g, \gamma_{ns}\omega\rp\right|<\fr{C_0}n, \forall x\in \fr1n\mbb Z^2.\label{sndif}}

Whenever $s,t\in \fr1n \mbb N$, $s\geq t$, simply define
\al{
S_n(s,t;g,\omega)(x):=g.
}

Next, we wish to interpolate $S_n$ with respect to the time variables $s$ and $t$. We want the interpolation to be continuous (even Lipschitz) in $s$ and $t$, but for this to make sense, we have to specify the image space and the topology on it.

\subsection{The space of continuous evolutions}\label{sec:contsemi}

Following  \cite{FR01}, we first define a general function space where we will embed the fixed-time evolutions of both the interpolated shuffling height process and the PDE.

\iffalse
First, we define a metric $d$ on $\Gamma$. Given $g_1, g_2\in \Gamma$, let
\al{
\|g_1-g_2\|_{l}&:=\sup_{|x|_\infty\leq l}|g_1(x)-g_2(x)|,\\
d(g_1,g_2)&:=\sum_{l=1}^\infty 2^{-l} \|g_1(x)-g_2(x)\|_l.
}

By definition of $\Gamma$, 
\ba{
\|g_1-g_2\|_{l}\leq |g_1(0)-g_2(0)|+2l,\label{411}
}
so the sum in $d(g_1,g_2)$ converges.  The triangle inequality is easy to see, and $d(g_1,g_2)=0$ iff $g_1=g_2$.

\fi

Given $g_1, g_2\in \Gamma$ and $k\in \mbb N\cup \lc \infty\rc$, let
\ba{
\|g_1-g_2\|_{k}&:=\sup_{|x|_1\leq k}|g_1(x)-g_2(x)|,\label{gammak}\\
d(g_1,g_2)&:=\sum_{i=1}^\infty 2^{-i}\|g_1-g_2\|_i.\label{gammametric}
}

By definition of $\Gamma$, 
\al{
\|g_1-g_2\|_{k}&\leq \sup_{|x|_\infty\leq k}|g_1(x)-g_2(x)|\\
&\leq |g_1(0)-g_2(0)|+\sup_{|x|_\infty\leq k}|g_1(x)-g_1(0)|+\sup_{|x|_\infty\leq k}|g_2(x)-g_2(0)|\\
&\leq |g_1(0)-g_2(0)|+4k\numberthis.
\label{411}
}
So $\|g_1-g_2\|_k$ grows at most linearly with respect to $k$, and the sum in \eqref{gammametric} converges. It is clear that $d(g_1,g_2)=0$ iff $g_1=g_2$, and the triangle inequality is easy to check, so \eqref{gammametric} defines a metric on $\Gamma$.

Let $\ag E^r_q$ ($r,q\geq 0$) denote the space of functions $F:\Gamma\rar \Gamma$ with the following properties:
\ba{
1.\;\;&F(g+m)=F(g)+m\text{ for every constant }m.\label{prop1}\\
2.\;\;&F(g_1)\leq F(g_2)\text{ whenever }g_1\leq g_2.\label{prop2}\\
3.\;\;&\sup_{g\in \Gamma}\|F(g)-g\|_0\leq r<\infty.\label{prop3}\\
4.\;\;&\text{If }g_1(x)=g_2(x)\text{ for all }x\text{ with }|x|_1\leq R,\text{ where } R\geq q, \text{ then } F(g_1)(x)=F(g_2)(x)\nonumber\\& \text{ for all } x \text{ with } |x|_1 \leq R-q. \label{prop4}
}

\begin{remn} Notice the similarities between these properties and \eqref{moveup}, Lemma~\ref{lem:monotone2}, Lemma~\ref{lem:vert}, Lemma~\ref{lem:linear}. \label{rem:simi}
\end{remn}

We can define a metric $D$ on $\ag E^r_q$. Given $F_1, F_2\in \ag E^r_q$, let
\al{
\|F_1-F_2\|_k&:=\sup_{g\in \Gamma}\|F_1(g)-F_2(g)\|_k,\\
D(F_1, F_2)&:= \sum_{i=1}^\infty 2^{-i} \fr{\|F_1-F_2\|_i}{1+\|F_1-F_2\|_i}.\numberthis\label{Ddef}
}

By \eqref{411}, 
\al{\|F_1(g)-F_2(g)\|_k&\leq |F_1(g)(0)-F_2(g)(0)|+4k\\
&=|(F_1(g)(0)-g(0))-(F_2(g)(0)-g(0))|+4k\\
&\leq 2r+4k
}
by Property~\eqref{prop3}. So $\|F_1-F_2\|_k$ is finite, and $D(F_1, F_2)$ is well defined. The triangle inequality is easy to check, and $D(F_1, F_2)=0$ iff $F_1(g)=F_2(g)$ for every $g\in \Gamma$.

There is a  ``localization'' lemma for functions in $\ag E$.

\begin{lem}[Localization]\label{lem:local}
Given any $F\in \ag E_q^r$, $g_1, g_2\in \Gamma$, $R\geq q$,
\al{
\|F(g_1)-F(g_2)\|_{R-q}\leq \|g_1-g_2\|_{R}.
}
In particular, $\|F(g_1)-F(g_2)\|_\infty \leq \|g_1-g_2\|_\infty$.
\end{lem}
\begin{proof}
We first show that $\Gamma$ is closed under $\vee$ and $\wedge$ operations. Given $g_1, g_2\in \Gamma$, $x,y\in \mbb R^2$, if $g_1(x)\geq g_2(x)$, $g_1(y)\geq g_2(y)$, then 
\al{
|(g_1\vee g_2) (x)-(g_1\vee g_2)(y)|=|g_1(x)-g_1(y)|\leq 2|x-y|_\infty.
}

If $g_1(x)\geq g_2(x)$, $g_1(y)\leq g_2(y)$, then
\al{
|(g_1\vee g_2) (x)-(g_1\vee g_2)(y)|=|g_1(x)-g_2(y)|.
}

Since
\al{
-2|x-y|_\infty\leq g_1(y)-g_1(x)\leq g_2(y)-g_1(x)\leq g_2(y)-g_2(x)\leq 2|x-y|_\infty,
}
we duduce that
\al{
|(g_1\vee g_2) (x)-(g_1\vee g_2)(y)|&\leq 2|x-y|_\infty.
}
The other two cases are similar, so we conclude that $\Gamma$ is closed under $\vee$. Taking negative shows closedness under $\wedge$. By Remark~\ref{rem:simi}, the rest of the proof proceeds in exactly the same way as the proof of Proposition~\ref{prop:local}.
\end{proof}

\begin{lem}\label{lem:lipg}
Functions $F$ in $\ag E_q^r$ are $2^q$-Lipschitz continuous.
\end{lem}
\begin{proof}
Suppose $g_1,g_2\in \Gamma$ satisfy that $d(g_1,g_2)\leq \delta$. By Lemme~\ref{lem:local}, 
\al{
d(F(g_1), F(g_2))&=\sum_{i=1}^\infty 2^{-i}\|F(g_1)-F(g_2)\|_i\\
&\leq \sum_{i=1}^\infty 2^{-i}\|g_1-g_2\|_{i+q}=2^q\sum_{i=1}^\infty 2^{-i-q}\|g_1-g_2\|_{i+q}\\
&\leq 2^q d(g_1,g_2).
}
\end{proof}
\begin{lem}\label{lem:polish}
The space $\ag E^r_q$ is compact. 
\end{lem}
\begin{proof}
The proof is  the same as  Lemma 3.2 in \cite{FR01}, so we omit the details. The main idea, to show totally-boundedness, is to choose a finite set of functions in $\ag E^r_q$, so that every $F\in \ag E^r_q$ can be approximated by at least one of them up to a required precision. Due to Property~\eqref{prop1}, a function $F\in \ag E_q^r$ is completely characterized by its image of $g\in \Gamma$ such that $g(0)=0$. Due to Lemma~\ref{lem:local} (which requires Properties~\eqref{prop1}, \eqref{prop2} and \eqref{prop4}) and the definition of the metric $D$, by specifying the input $g$ and output $F(g)$ on a finite region around the origin, we can approximate nearby functions in $\ag E_q^r$ with a controllable error. Furthermore, Property~\eqref{prop3} provides a finite bound to the possible range of $F(g)$. Combined with the fact that both $g$ and $F(g)$ are in $\Gamma$,  only a finite number of functions in  $\ag E^r_q$ are required. 
\end{proof}

Now let $\ag C_T:=C\lp[0,T]\times [0,T];\ag E^r_q\rp=C\lp[0,T];C\lp [0,T];\ag E^r_q\rp\rp$ be the space of continuous functions $S:[0,T]\times [0,T]\rar \ag E^r_q$, with the uniform topology, where the distance between two functions $R_1, R_2\in \ag C_T$ is given by
\ba{
\rho(R_1,R_2):=\sup_{s,t\in [0,T]^2} D\lp R_1(s,t), R_2(s,t)\rp.\label{ctmetric}
}

It is well known  that if $Y$ is a Polish space (separable complete metric space), then $C([0,T];Y)$ is also a Polish space (see for example \cite[Theorem 4.19]{AK12}). Then since $\ag E^r_q$ is compact, it is in particular Polish. Thus $\ag C_T$ is also Polish.

\subsection{Smoothing out the height process temporally}\label{sec:smoothtime}
Having the abstract space set up, we return to the function $S_n(s,t;g,\omega)(x)$ defined previously. When $s<t$, the function $S_n(s,t;\cdot,\omega)$ does not belong to $\ag E^r_q$ since the value of $S_n(s,t;g,\omega)(2x)$ at $x\in \fr1n \mbb Z^2$ belongs to $\fr1n \mbb Z$, while the constant $m$ in Property~\eqref{prop1} of $\ag E^r_q$ is arbitrary, which is used in the proof of Lemma~\ref{lem:local} and  Lemma~\ref{lem:polish}. To force Property~\eqref{prop1}, we could consider a modified function $g\mapsto S_n(s,t; g-g(0),\omega)+g(0)$, but then one can check that Property~\eqref{prop2} no longer holds. To guarantee both Property~\eqref{prop1} and \ref{prop2}, we consider the following modification:
\al{
\wh S_n(s,t;g,\omega):=\fr 14\int_{u=0}^{ 4}\lb S_n\lp s,t; g-g(0)-\fr u n,\omega\rp+\fr un\rb du +g(0).
}
Roughly speaking, we are averaging the result of $S_n$ across a vertical range of $\fr 4n$. It is obvious that $\wh S_n(s,t;\cdot, \omega)$ still maps $\Gamma$ to $\Gamma$.

First, we prove a couple of lemmas, which will be used to show that $\wh S_n$ is in $\ag E_q^r$, and also later in the paper.

\begin{lem}\label{lem:snvert}
For any $s,t\in \fr1n \mbb N$ such that $s<t$, and $g\in \Gamma$, we have 
\al{
\|S_n(s,t;g,\omega)-g\|_\infty\leq \fr{6+C_0}n+4(t-s).
}
\end{lem}
\begin{proof}
From \eqref{g2hbd}, it follows that $|ng(x)-\vph_g^n(x)|\leq 4$ for all $x\in \mbb Z^2$. And by Lemma~\ref{lem:vert}, for all $x\in\fr1n \mbb Z^2$, 
\al{
\left|\vph_g^n(x)-h(nx, n(t-s);\vph_g^n,\gamma_{ns}\omega)\right|\leq 4n(t-s).
}
Combined with \eqref{sndif}, we get for all $x\in \fr1n\mbb Z^2$,
\al{
\left|S_n(s,t;g,\omega)(x)-g(x)\right|\leq \fr{4+C_0}n+4(t-s).
}
Finally, since $S_n$ and $g$ are both in $\Gamma$, we deduce that for all $x\in \mbb R^2$,
\al{
\left|S_n(s,t;g,\omega)(x)-g(x)\right|\leq \fr{4+C_0+2}n+4(t-s).
}
\end{proof}

\begin{lem}\label{lem:snlinear}
Given $x\in\mbb R^2$ and $g_1,g_2\in \Gamma$, if $g_1(y)=g_2(y)$ for all $y$ such that $|y-x|_1\leq R$, then $S_n(s,t;g_1,\omega)(y)=S_n(s,t;g_2,\omega)(y)$ for all $y$ such that $|y-x|_1\leq R-2(t-s)-\fr9n$.
\end{lem}

\begin{proof}
Assuming that $g_1(y)=g_2(y)$ for all $y\in\mbb R^2$ such that $|y-x|_1\leq R$,  by Proposition~\ref{prop:locdef}, we have $\vph^n_{g_1}(z)=\vph^n_{g_2}(z)$ for all $z\in \mbb Z^2$ with $|\fr zn-x|_1\leq R-\fr1n$. This is equivalent to
\al{
|z-nx|_1=|(z-\lfloor nx\rfloor)+(\lfloor nx\rfloor- nx)|_1\leq nR-1.
}
Since $|\lfloor nx\rfloor- nx|_1\leq 2$, we know that $\vph^n_{g_1}(z)=\vph^n_{g_2}(z)$ for all $z\in \mbb Z^2$ with  $|z-\lfloor nx\rfloor|_1\leq nR-3$. 

 Applying Lemma~\ref{lem:linear}, we have 
\al{h(ny,nt-ns;\vph^n_{g_1},\gamma_{ns}\omega)= h(ny,nt-ns;\vph^n_{g_2},\gamma_{ns}\omega)}
for all $y\in \fr1n\mbb Z^2$ such that $|ny-\lfloor nx\rfloor|_1\leq nR-3-2n(t-s)$. By construction, the value of $S_n$ at a point $y\in\mbb R^2$ is determined by the value of $h$ at $(2i,2j), (2i+2, 2j),(2i+2,2j+2), (2i, 2j+2)$ such that $i,j\in \mbb Z$, and the square formed by these four points contains $ny$. Therefore, $S_n(s,t;g_1,\omega)(y)= S_n(s,t;g_2,\omega)(y)$ for all $y\in \mbb R^2$ such that 
\al{|ny-\lfloor nx\rfloor|_1&\leq nR-3-2n(t-s)-4\\
\left|n(y-x)+(nx-\lfloor nx\rfloor)\right|_1&\leq nR-7-2n(t-s).}
Since $\left|nx-\lfloor nx\rfloor\right|_1\leq 2$, we conclude that $S_n(s,t;g_1,\omega)(y)=S_n(s,t;g_2,\omega)(y)$ for all $y$ such that 
\al{
\left|n(y-x)\right|_1&\leq nR-7-2n(t-s)-2\\
\left|y-x\right|_1&\leq R-2(t-s)-\fr9n.
}
\end{proof}

\begin{prop}\label{prop:whsninerq}
Given $T>0$, there exist universal constants $r, q>0$ such that $\wh S_n(s,t;\cdot,\omega)\in \ag E^r_q$, where $r, q$ will be specified below, for all $n\in \mbb Z_{>0}$ and $s,t\in [0,T]\cap \fr1n\mbb N$.
\end{prop} 

\begin{proof}
We assume that $s<t$, for otherwise $S_n(s,t;\cdot,\omega)$ and thus $\wh S_n(s,t;\cdot,\omega)$ is simply identity.

Property~\eqref{prop1} holds  because for any constant $m$,
\al{
\wh S_n(s,t; g+m,\omega)&=\fr 14\int_{u=0}^{ 4}\lb S_n\lp s,t; g+m-g(0)-m-\fr un,\omega\rp+\fr un\rb du +g(0)+m\\
&=\wh S_n(s,t; g, \omega)+m.
}

From Lemma~\ref{lem:snvert}, we know that $|S_n(s,t;g,\omega)(0)-g(0)-u|\leq (6+C_0+4)/n+4(t-s)\leq 10+C_0+4T$ for all $u\in [0,\fr n4)$, so by taking the integral over $u$, Property~\eqref{prop3} holds for $\wh S_n(s,t;\cdot,\omega)$ with $r=10+C_0+4T$. 

We confirm that Property~\eqref{prop4} holds with $q=2T+9$ by specializing Lemma~\ref{lem:snlinear} to $x=0$.

Finally, to verify Property~\eqref{prop2}, suppose $g_1, g_2\in \Gamma$ satisfy that $g_1\leq g_2$. We will construct a measure-preserving bijection $\eta:[0,4)\rar [0,4)$. We set up some notations for convenience:
\al{
\delta_0:=g_2(0)-g_1(0),\;\;\;\;&\delta:= \inf_{x\in \mbb R^2}\lc g_2\lp x\rp- g_1\lp x\rp \rc \geq 0\in [0, +\infty),\\
f_1:=g_1-g_1(0),\;\;\;\;
&f_2:=g_2-g_2(0).
}

Let us consider
\al{
&\lb S_n\lp s,t; g_2-g_2(0)-\fr u n,\omega\rp+\fr un + g_2(0)\rb - \lb S_n\lp s,t; g_1-g_1(0)-\fr {\eta(u)} n,\omega\rp+\fr {\eta(u)}n + g_1(0)\rb\\
=& S_n\lp s,t;f_2-\fr u n,\omega\rp - S_n\lp s,t;f_1-\fr {\eta(u)} n,\omega\rp+\fr {u- \eta(u)}n + \delta_0.\numberthis\label{4.13.1} 
}
We define $\eta(u)$ to be the unique element in $[0,4)$ such that $\delta-\delta_0+\fr {\eta(u)-u}n$ is an integer multiple of $\fr4n$, say $\fr{4k}n$ for some $k\in \mbb Z$.

By construction of $S_n$,
\ba{
S_n\lp s,t;f_2-\fr u n,\omega\rp&=S_n\lp s,t;f_2-\fr u n-\fr{4k}n,\omega\rp + \fr{4k}n,\label{4.13.2}
}
and
\al{
\lp f_2-\fr u n-\fr{4k}n\rp -\lp f_1-\fr {\eta(u)} n\rp &= g_2-g_1-(g_2(0)-g_1(0))+\fr {\eta(u)-u}n-\fr{4k}n\\&\geq \delta-\delta_0+\fr {\eta(u)-u}n-\fr{4k}n =0.
}

Therefore again by construction of $S_n$ and Lemma~\ref{lem:monotone2},
\ba{
S_n\lp s,t;f_2-\fr u n-\fr{4k}n,\omega\rp - S_n\lp s,t;f_1-\fr {\eta(u)} n,\omega\rp\geq 0.\label{4.13.3}
}

Putting \eqref{4.13.1}, \eqref{4.13.2},\eqref{4.13.3} together, we get that
\al{
\text{LHS of \eqref{4.13.1}}\geq \fr{4k}n+\fr {u- \eta(u)}n + \delta_0=\delta\geq 0.
}
Now integrating LHS of \eqref{4.13.1} over $u\in [0,4)$, we conclude that $\wh S_n(s,t;g_1,\omega)\leq \wh S_n(s,t;g_2,\omega)$.
\end{proof}

So far $\wh S_n(s,t;\cdot,\omega)\in \ag E^r_q$ is only defined for $s,t\in \fr1n\mbb N$. We would like to interpolate it to a Lipschitz continuous function in $(s,t)\in [0,T]^2$ with a universal Lipschitz constant (taking $\ell_1$ metric on $[0,T]^2$), and treat it as an element of $\ag C_T$.  We shall first check the Lipschitz continuity of $\wh S_n(s,t;\cdot,\omega)$ on $(s,t) \in \fr1n\mbb N^2$, and then interpolate bilinearly to the entire $[0,T]^2$. It suffices to consider $S_n$ since $\wh S_n$ is just an average of $S_n$.

\begin{prop}\label{prop:discrlip}
For all $n\in\mbb Z_{>0}$ , $(s,t)\in \fr1n\mbb N^2$, $g\in \Gamma$, 
\al{
\left\| S_n(s,t;g,\omega)- S_n\lp s,t+\fr1n;g,\omega\rp\right\|_\infty&\leq \fr {10+2C_0}n,\\
\left\| S_n(s,t;g,\omega)- S_n\lp s+\fr1n,t;g,\omega\rp\right\|_\infty&\leq \fr {10+2C_0}n.
} 
\end{prop}
\begin{proof}
 To check Lipschitz continuity with respect to $t$, take $F_1=S_n(s,t;\cdot,\omega)$ amd $F_2=S_n\lp s,t+\fr1n;\cdot,\omega\rp$ where $(s,t)$ and $\lp s,t+\fr1n\rp$ are both in $\fr1n\mbb Z_{\geq 0}^2$. When $s=t$, $F_1(g)=g$. Then Lemma~\ref{lem:snvert} implies that
\al{
\| F_1(g)- F_2(g)\|_\infty&=\|g- F_2(g)\|_\infty\leq \fr{6+C_0}n+\fr4n=\fr{10+C_0}n,\;\;\;\forall g\in \Gamma.}
When $s<t$, given an input $g$, $F_1$ and $F_2$ are computed from the same $\vph_g^n$. The outputs were interpolated from $h( n x, n(t-s); \vph^n_g, \gamma_{ns}\omega)$ and $h( n x, n(t-s)+1; \vph^n_g, \gamma_{ns}\omega)$ respectively. Applying \eqref{semi} and Lemma~\ref{lem:vert}, these two height functions differ by at most $4$ everywhere. Thus by \eqref{sndif}, $\| F_1(g)- F_2(g)\|\leq (4+2C_0+2)\fr1n$. 
When $s>t$ both functions are identity. Combining the bounds, we obtain the first inequality.

Checking Lipschitz continuity with respect to $s$ is somewhat different. This time, take $F_1=S_n(s,t;\cdot,\omega)$ amd $F_2=S_n\lp s+\fr1n,t;\cdot,\omega\rp$ where $(s,t)$ and $\lp s+\fr1n,t\rp$  both lie in $\fr1n\mbb N^2$. If $s+\fr1n = t$, $F_2(g)=g$, so  again Lemma~\ref{lem:snvert} implies that $\| F_1(g)- F_2(g)\|_\infty\leq (10+C_0)\fr1n$. When $s+\fr1n<t$, $F_1(g)$ and $F_2(g)$ are interpolated from $h( n x, n(t-s); \vph^n_g, \gamma_{ns}\omega)$ and $h( n x, n(t-s) -1; \vph^n_g, \gamma_{ns+1}\omega)$ respectively. By \eqref{semi},
\al{
h( n x, n(t-s); \vph^n_g, \gamma_{ns}\omega)&=h( n x, n(t-s)-1;h(\cdot, 1;\vph^n_g, \gamma_{ns}\omega), \gamma_{ns+1}\omega),
} 
and furthermore by Lemma~\ref{lem:vert},
\al{
|h(nx, 1;\vph^n_g, \gamma_{ns}\omega)-\vph^n_g(nx)|\leq 4,\;\;\;\;\; \forall x\in \fr1n\mbb Z^2.
} 
Combined with Proposition~\ref{prop:local}, we obtain that
\al{
|h( n x, n(t-s); \vph^n_g, \gamma_{ns}\omega)-h( n x, n(t-s) -1; \vph^n_g, \gamma_{ns+1}\omega)|\leq 4 ,\;\;\;\;\; \forall x\in \fr1n\mbb Z^2.
}
Finally using \eqref{sndif}, we conclude that $\| F_1(g)- F_2(g)\|_\infty\leq (4+2C_0+2)\fr1n$. When $s\geq t$, both functions are identity. As a result, we get the second inequality.
\end{proof}

The last step is to interpolate $\wh S_n$ to continuous time bilinearly. To be more specific, for each $(s,t)\in \fr1n \mbb N^2$ such that $0\leq s,t\leq T-\fr1n$, and each $g\in \Gamma$, let
$$\begin{array}{rll}
& f_{00}=\wh S_n(s,t;g,\omega), &f_{10}=\wh S_n\lp s+\fr1n,t;g,\omega\rp,\\
&f_{01}=\wh S_n\lp s,t+\fr1n;g,\omega\rp, &f_{11} =\wh S_n\lp s+\fr1n, t+\fr1n;g,\omega\rp.
\end{array}$$
Then for each $(x,y)\in [0,1]^2$, define
\al{
\wh S_n\lp s+\fr xn,t+\fr yn;g,\omega\rp&:=f_{00}+(f_{10}-f_{00})x+(f_{01}-f_{00})y+(f_{00}+f_{11}-f_{01}-f_{10})xy.
}
It is easy to see that $\wh S_n(s,t;g,\omega)$ assumes the original values at $(s,t)\in \fr1n\mbb N^2$, and is linear in $(s,t)$ whenever $s$ or $t$ is constant, with slopes bounded by the constant $10+2C_0$ from Proposition~\ref{prop:discrlip}. In particular, $\forall (s_1,t_1), (s_2,t_2)\in \mbb R^2$,  
\ba{\label{ptlip}
\| \wh S_n(s_1,t_1;g,\omega)- \wh S_n\lp s_2, t_2;g,\omega\rp\|_\infty&\leq (10+2C_0)(|s_1-s_2|+|t_1- t_2|).
}
Moreover, this linear property shows  that  $\wh S_n(s,t;\cdot, \omega)$ is in $\ag E^r_q$ for all $(s,t)\in [0,T]^2$. To summarize, we have the following statement.
\begin{prop}\label{prop:lip}
For all $n\in \mbb Z_{>0}$, $\wh S_n(s,t;\cdot, \omega)$ as a function of $(s,t)\in [0,T]^2$ is an element of $\ag C_T$ with Lipschitz constant $10+2C_0$, where $[0,T]^2$ is equipped with $\ell_1$ metric, i.e. \al{
D\lp \wh S_n(s_1,t_1;\cdot,\omega), \wh S_n\lp s_2,t_2;\cdot,\omega\rp\rp&\leq (10+2C_0)(|s_1-s_2|+|t_1- t_2|).
} 
\end{prop}

Even though we will define the limit using $\wh S_n$, we are not deviating much from $S_n$ based on the following proposition. 
\begin{prop}\label{prop:error} For all $(s,t)\in \fr1n \mbb N^2$,
\ba{
\left\| \wh S_n(s,t;g,\omega)-S_n\lp s,t;g,\omega\rp\right\|_\infty\leq \fr{18+2C_0}n.
}
\end{prop}
\begin{proof}
Suppose $k$ is the largest integer such that 
$4k/n\leq g(0)$, and given any $u\in [0,4)$, we have $\vph_{g-g(0)-u/n}^n+4k=\vph_{g-g(0)-u/n+4k/n}^n$. By \eqref{moveup} and the construction of $S_n$,  for all $(s,t)\in \fr1n \mbb N^2$,
\ba{S_n\lp s,t; g-g(0)-\fr un,\omega\rp+\fr{4k}n=S_n \lp s,t;g-g(0)-\fr un+\fr{4k}n, \omega\rp.\label{4.16.1}}
Since $|4k/n-g(0)-u/n|\leq 4/n$, by \eqref{g2hbd}, $|\vph_{g-g(0)-u/n+4k/n}^n-\vph_g^n|\leq 4+4+4=12$. Then as before, invoking Proposition~\ref{prop:local} and \eqref{sndif}, we deduce that $\forall (s,t)\in \fr1n \mbb N^2$
\ba{
\left\|S_n\lp s,t;g-g(0)-\fr un+\fr{4k}n,\omega\rp-S_n\lp s,t; g,\omega\rp\right\|_\infty\leq \fr1n\lp12+2C_0+2\rp.\label{4.16.2}
}
On the other hand, by definition
\al{
\wh S_n(s,t;g,\omega)&=\fr 14\int_{u=0}^{ 4}\lb S_n\lp s,t; g-g(0)-\fr u n,\omega\rp+\fr{4k}n+\lp\fr un+g(0)-\fr{4k}n\rp\rb du, 
}
so combining \eqref{4.16.1} and \eqref{4.16.2}, we have 
\al{
\left\| \wh S_n(s,t;g,\omega)-S_n\lp s,t;g,\omega\rp\right\|_\infty\leq \fr1n\lp12+2C_0+2+4\rp.
}
\end{proof}

\section{Limit points}\label{sec:lp}

\subsection{Precompactness}

Throughout the previous section, we kept the Bernoulli mark $\omega\in \Omega$ fixed. From now on we shall work with the whole probability space $\Omega$ again. Then $\wh S_n$ becomes a random variable on $\Omega$, and induces a probability measure $\mu_n$ on $\ag C_T$. Since $\ag C_T$ is Polish, by Prokhorov's theorem, the family $\lc \mu_n\rc$ is tight iff $\lc \mu_n\rc$ is precompact. Recall the definition that the sequence $\lc \mu_n\rc$ is tight if for any $\vep>0$, there exists a compact set $K_\vep$ such that $\mu_n(K_\vep)\geq 1-\vep$ for all $n$. On the other hand, $\lc \mu_n\rc$ is precompact if its closure is sequentially compact, that is, every subsequence of $\lc \mu_n\rc$ further contains a weakly convergent subsequence. See \cite{BP13} for more background.

Let $\ag C'_T$ denote the subset of $\ag C_T$ which consists of $(10+2C_0)$-Lipschitz continuous functions on $[0,T]^2$ equipped with the $\ell_1$ metric. By Proposition~\ref{prop:lip}, $\mu_n(\ag C'_T)=1$ for all $n$, so to obtain the tightness of $\lc \mu_n\rc$,  it suffices to show that $\ag C'_T$ is compact. We shall use a generalized version of the Arzel\`a-Ascoli theorem (see for example \cite[Theorem~7.17]{KJ17}).

\begin{thm}A subset $E$ of the space of continuous function from a compact Hausdorff space  $X$ to a metric space $Y$ with uniform topology is compact iff the following three conditions hold:
\begin{enumerate}
\item{$E$ is closed,}
\item{$E(x)$ has a compact closure for every $x\in X$,}
\item{$E$ is equicontinuous.}
\end{enumerate}
\end{thm}

Applied to $\ag C_T$, the subset $\ag C'_T$ is equicontinuous because it consists of $(10+2C_0)$-Lipschitz continuous functions. The second condition is given for free, because $\ag E_q^r$ is compact. To check that $\ag C'_T$ is closed, suppose a sequence $(F_n)_{n\in \mbb N}$ in $\ag C'_T$ converges uniformly to $F$ in $\ag C_T$. Let $d_1$ denote the $\ell_1$ distance on $[0,T]^2$. Given $x,y\in [0,T]^2$, we have $D(F_n(x), F_n(y))\leq (10+2C_0)d_1(x,y)$ since $F_n$ is in $\ag C'_T$. Due to convergence, given $\vep>0$, there exists $N$ such that for all $n>N$, $D(F_n(x), F(x))<\vep$ and $D(F_n(y), F(y))<\vep$. Then by triangle inequality, for $n>N$, 
\al{
D(F(x),F(y))&\leq D(F(x), F_n(x))+D(F_n(x), F_n(y)) + D(F_n(y), F(y))\\
&< 2\vep + (10+2C_0)d_1(x,y).
}
As $\vep>0$ is arbitrary, $D(F(x), F(y))\leq (10+2C_0)d_1(x,y)$, and thus $F$ is also in $\ag C'_T$.

\subsection{Characterization of limit points}\label{sec:limit}

Now we know that the closure of $\lc \mu_n\rc$ is sequentially compact. In order to identify the subsequential limit points of $\lp \mu_n\rp_{n\in\mbb Z_{>0}}$ as semigroups of Hamilton-Jacobi equations, we shall first make some characterization.

Since $\ag C'_T$ is compact, and $\mu_n(\ag C'_T)=1$, we will restrict to $\ag C'_T$ from now. Let $\ag C^0_T\subset \ag C'_T$ be the subset of elements $S$  that satisfy the following additional conditions:
\ba{
1.\;\;&\text{For all }g\in \Gamma, S(t_1,t_2;g)=g\text{ when }t_1\geq t_2,\label{cond1}\\
2.\;\;&\text{For all }g\in \Gamma\text{ and }t_1, t_2, t_3\text{ such that }0\leq t_1\leq t_2\leq t_3\leq T, S(t_2,t_3;S(t_1,t_2;g))=\nonumber\\
&S(t_1,t_3;g),\label{cond2}\\
3.\;\;&\text{Given }x\in \mbb R^2\text{ and }g_1,g_2\in \Gamma,\text{ if }g_1(y)=g_2(y)\text{ for all }y\text{ such that }|y-x|_1 \leq R,\nonumber\\ 
&\text{then }S(s,t;g_1)(y)=S(s,t;g_2)(y)\text{ for all }y\text{ such that }|y-x|_1\leq R-2(t-s)\text{ and}\nonumber\\&\text{all }s,t\text{ such that }0\leq s< t\leq T.\label{cond3}
}

\iffalse
\begin{enumerate}
\item{For all $g\in \Gamma$, $S(t_1,t_2;g)=g$ when $t_1\geq t_2$,\label{cond1}}
\item{For all $g\in \Gamma$ and $t_1, t_2, t_3$ such that $0\leq t_1\leq t_2\leq t_3\leq T$, $S(t_1,t_3;g)=S(t_2,t_3;S(t_1,t_2;g))$,\label{cond2}}
\iffalse{\item{\label{cond2}$S(s,t;g)$ is $(10+2C_0)$-Lipschitz continuous with respect to $(s,t)$ on $[0,T]^2$ equippied with $\ell_1$ metric,}
\item{\label{cond3} For all $g\in \Gamma$, $s,t$ such that $0\leq s< t<t+\delta\leq T$, $S(s,t;g)=S(s+\delta,t+\delta;g)$,}
\item{\label{cond4} For all $g\in \Gamma$, $s,t$ such that $0\leq s<t\leq T$, $y\in\mbb R^2$, $\tau_y(S(s,t;g))=S(s,t;\tau_yg)$, where $\tau_yg(x):=g(x-y)$. }}\fi
\item{\label{cond3}Given $x\in \mbb R^2$ and $g_1,g_2\in \Gamma$, if $g_1(y)=g_2(y)$ for all $y$ such that $|y-x|_1 \leq R$, then $S(s,t;g_1)(y)=S(s,t;g_2)(y)$ for all $y$ such that $|y-x|_1\leq R-2(t-s)$ and all $s,t$ such that $0\leq s< t\leq T$,}\iffalse{
\item{There exists a continuous function $H_S:U\rar \mbb R$ depending on $S$ such that $S(s,t;g_p)=g_p-(t-s)H_S(p)$, where $U=\lc |s|+|t|\leq 2\rc$, $p\in U$, and $g_p(x)=x\cdot p$. \label{cond8}}}\fi
\end{enumerate}
\fi

\begin{prop}\label{prop:char}
All subsequential limits of $\lp \mu_n\rp_{n\in\mbb Z_{>0}}$ lie in $\ag C^0_T$ with probability 1. 
\end{prop}
When we say {\it subsequential limits}, we mean the limit of a weakly convergent subsequence $(\mu_{n_i})_{i\in \mbb N}$ such that $(n_i)_{i\in \mbb N}$ is a strictly increasing sequence of positive integers.

\begin{proof}
Suppose $\lp \mu_{n_i}\rp_{i\in \mbb N}$ is such a weakly convergent subsequence. For convenience, we will denote this subsequence by $\lp\mu_i\rp_{i\in \mbb N}$ where $\mu_i:=\mu_{n_i}$. Also denote the limiting measure by $\mu$. By Portmanteau theorem (see for example \cite[Theorem~2.1]{BP13}), the weak convergence is equivalent to
\al{
\limsup_{i\rar \infty} \mu_{i}(E)\leq \mu(E)
}
for all closed subset $E$ of $\ag C'_T$.

Let $E_1$ denote the subset of $\ag C'_T$ that satisfies Condition~\eqref{cond1}. Then clearly $E_1$ is a closed set, as uniform convergence in $\ag C'_T$ implies pointwise convergence. Also by definition we have $\mu_i(E_1)=1$, so $\mu(E_1)=1$ as well.

Now we turn to Condition~\eqref{cond2}, which describes a semigroup property. If  $E_2$ denotes the subset with such property, due to the discrete nature and interpolation, it is unlikely that any $\mu_i$ has probability 1 on $E_2$, and it is unclear what is $\limsup \mu_i(E_2)$. To work around this, given $\vep> 0$, let $E_2^\vep$ denote the set of $S\in \ag C'_T$ that satisfies the following condition: 
\al{D(S(t_1,t_3;\cdot), S(t_2,t_3;S(t_1,t_2;\cdot)))\leq \vep,\;\;\;\;\;\forall t_1,t_2,t_3 \text{ such that } 0\leq t_1\leq t_2\leq t_3\leq T.} 

Here the distance $D(\cdot, \cdot)$ is defined by \eqref{Ddef}. The expression is well defined because $S(t_2,t_3;S(t_1,t_2;\cdot))$ is obviously in $\ag E_{2q}^{2r}$.

\begin{lem}\label{f2eclosed}
The subset $E_2^\vep$ is closed in $\ag C'_T$.
\end{lem}
\begin{proof}Suppose $(F_n)_{n\in \mbb N}$ in $E_2^\vep$ converges  to $F\in \ag C'_T$, and take $t_1,t_2,t_3$ such that $0\leq t_1\leq t_2\leq t_3\leq T$. Given some large $k\in \mbb Z_{>0}$ and small $\delta>0$, define 
\al{a(k,\delta)= 2^{-k}\fr{\delta}{1+\delta}.} 
Then there exists $N_a$ such that 
\al{D(F_n(s,t), F(s,t))<a(k,\delta),\;\;\;\; \forall n>N_a, (s,t)\in [0,T]^2.} 
For such $n$, in particular we have $D(F_n(t_1,t_2), F(t_1,t_2))<a(k,\delta)$. This implies that
\ba{
\sup_{g\in \Gamma}\|F_n(t_1,t_2;g)-F(t_1,t_2;g)\|_k\leq \delta.\label{t12ub}
}

We also have  $D(F_n(t_2,t_3), F(t_2,t_3))<a(k,\delta)$, which tells us that
\ba{
\sup_{g\in \Gamma}\|F_n(t_2,t_3;F_n(t_1,t_2;g))-F(t_2,t_3;F_n(t_1,t_2;g))\|_k\leq \delta.\label{t231}
}
By Lemma~\ref{lem:local}, if $k\geq q$, then $\forall g\in \Gamma$,
\ba{
\|F(t_2,t_3;F_n(t_1, t_2;g))-F(t_2,t_3;F(t_1, t_2;g))\|_{k-q}\leq \|F_n(t_1, t_2;g)-F(t_1, t_2;g)\|_k\leq \delta,\label{t232}
}
where the last inequality is due to \eqref{t12ub}. From \eqref{t231} and \eqref{t232}, we deduce that
\al{
D(F_n(t_2,t_3;F_n(t_1,t_2;\cdot)), F(t_2,t_3;F(t_1,t_2;\cdot)))&\leq 2\delta+2^{-k+q}.
}
Combined with $D(F_n(t_1,t_3), F(t_1,t_3))<a(k,\delta)$ and the fact that $F_n\in F_2^\vep$, we have
\al{
&D(F(t_1,t_3;\cdot), F(t_2,t_3;F(t_1,t_2;\cdot)))\leq D(F(t_1,t_3;\cdot), F_n(t_1,t_3;\cdot))+ \\
&+D(F_n(t_1,t_3;\cdot), F_n(t_2,t_3;F_n(t_1,t_2;\cdot))) + D(F_n(t_2,t_3;F_n(t_1,t_2;\cdot)), F(t_2,t_3;F(t_1,t_2;\cdot)))\\
&\leq a(k,\delta)+\vep+2\delta+2^{-k+q}.}
By taking $k\rar \infty$ and $\delta\rar 0$ simultaneously, we conclude that \al{D(F(t_1,t_3;\cdot), F(t_2,t_3;F(t_1,t_2;\cdot)))\leq \vep,} so $E_2^\vep$ is indeed closed. 
\end{proof}

Now we want to show that $\limsup_{i\rar \infty} \mu_i(E_2^\vep)=1$. In fact, we claim the following is true.
\begin{lem}\label{f2ebign}
There exists $N>0$ such that for all $n>N$, $\wh S_n\in E_2^\vep$ with probability 1. 
\end{lem}
\begin{proof}
By definition \eqref{Ddef}, it suffices to show that, when $n$ is large enough, for all $t_1,t_2,t_3$ such that $0\leq t_1\leq t_2\leq t_3\leq T$, $\omega\in \Omega$, $g\in \Gamma$,
\ba
{\|\wh S_n(t_1,t_3;g,\omega)-\wh S_n(t_2,t_3;\wh S_n(t_1,t_2;g,\omega),\omega)\|_\infty\leq \vep.
}
Let $t'_i=\lfloor nt_i \rfloor \fr1n$ for $i=1,2,3$. Then by \eqref{ptlip}, for any $g'\in \Gamma$,  $i,j\in \lc 1,2,3\rc$ such that $i<j$, 
\al{
&\|\wh S_n(t_1, t_2;g',\omega)-S_n(t'_i,t'_j;g',\omega)\|_\infty\\
\leq& \|\wh S_n(t_i,t_j;g',\omega)- \wh S_n(t'_i,t'_j;g',\omega)\|_\infty+ \|\wh S_n(t'_i,t'_j;g',\omega)- S_n(t'_i,t'_j;g',\omega)\|_\infty\\
\leq& (10+2C_0)\fr2n+\fr{18+2C_0}n = \fr{38+6C_0}n.\numberthis\label{snapprox}
}
where the second inequality uses \eqref{ptlip} and Proposition~\ref{prop:error}.

Therefore, applying Lemma~\ref{lem:local},
\al{
&\|\wh S_n(t_2, t_3; \wh S_n(t_1, t_2;g,\omega),\omega)-\wh S_n(t_2, t_3; S_n(t'_1,t'_2;g,\omega),\omega)\|_\infty\\
\leq &\|\wh S_n(t_1, t_2;g,\omega)-S_n(t'_1,t'_2;g,\omega)\|_\infty
\leq \fr{38+6C_0}n.\numberthis\label{snapprox2}
}

Combining \eqref{snapprox} and \eqref{snapprox2}, we can bound
\al{
&\|\wh S_n(t_1,t_3;g,\omega)-\wh S_n(t_2,t_3;\wh S_n(t_1,t_2;g,\omega),\omega)\|_\infty\leq \|\wh S_n(t_1,t_3;g,\omega)- S_n(t'_1,t'_3;g,\omega)\|_\infty+\\
&+\|\wh S_n(t_2, t_3; \wh S_n(t_1, t_2;g,\omega),\omega)-\wh S_n(t_2, t_3; S_n(t'_1,t'_2;g,\omega),\omega)\|_\infty+\\&+ \|\wh S_n(t_2, t_3; S_n(t'_1,t'_2;g,\omega),\omega)
-S_n(t'_2, t'_3; S_n(t'_1,t'_2;g,\omega),\omega)\|_\infty+\\
&+\|S_n(t'_1,t'_3;g,\omega)-S_n(t'_2,t'_3;S_n(t'_1,t'_2;g,\omega),\omega)\|_\infty\\
&\leq 3\fr{38+6C_0}n+\|S_n(t'_1,t'_3;g,\omega)-S_n(t'_2,t'_3;S_n(t'_1,t'_2;g,\omega),\omega)\|_\infty.
}
Suppose $n$ is large enough so that $3(38+6C_0)/n<\vep/2$, then we are left to show that 
\ba{
\|S_n(t'_1,t'_3;g,\omega)-S_n(t'_2,t'_3;S_n(t'_1,t'_2;g,\omega),\omega)\|_\infty&\leq \vep/2.
}

By \eqref{sndif}, for all $g'\in \Gamma$ and $x\in \fr1n \mbb Z^2$, $i,j\in\lc 1,2,3\rc$ such that $i<j$,
\ba{
\left|S_n(t'_i,t'_j;g',\omega)(x)-\fr1n h\lp n x, n(t'_j-t'_i); \vph^n_{g'}, \gamma_{nt'_i}\omega\rp\right|<\fr{C_0}n.\label{snh}
}
Let
\al{
\lambda_g:=\vph^n_{S_n(t'_1,t'_2;g,\omega)}.
}
Then by \eqref{g2hbd}, for all $x\in \fr1n \mbb Z^2$,
\ba{
\left| \fr1n\lambda_g(nx)- S_n(t'_1, t'_2;g,\omega)(x)\right|\leq \fr4n.\label{snvph}
}
Combining \eqref{snh} and \eqref{snvph} yields
\ba{
\left|\lambda_g(nx)-  h\lp n x, n(t'_2-t'_1); \vph^n_g, \gamma_{nt'_1}\omega\rp\right|\leq 4+C_0,\;\;\;\forall x\in \fr1n \mbb Z^2.\label{hlamb}
}

For all $x\in \fr1n \mbb Z^2$,
\al{
&\left|S_n(t'_1,t'_3;g,\omega)(x)-S_n(t'_2,t'_3;S_n(t'_1,t'_2;g,\omega),\omega)(x)\right|\\
\leq &\left|S_n(t'_1,t'_3;g',\omega)(x)-\fr1n h\lp n x, n(t'_3-t'_1); \vph^n_g, \gamma_{nt'_1}\omega\rp(x)\right|+\\
&+\left|S_n(t'_2,t'_3;S_n(t'_1,t'_2;g,\omega),\omega)(x)-\fr1n h\lp n x, n(t'_3-t'_2); \lambda_g, \gamma_{nt'_2}\omega\rp(x)\right|+\\
&+\left|\fr1n h\lp n x, n(t'_3-t'_2); \lambda_g, \gamma_{nt'_2}\omega\rp(x)-\fr1n h\lp n x, n(t'_3-t'_2);  h\lp n x, n(t'_2-t'_1); \vph^n_g, \gamma_{nt'_1}\omega\rp, \gamma_{nt'_2}\omega\rp(x)\right|+\\
&+\left|\fr1n h\lp n x, n(t'_3-t'_1); \vph^n_g, \gamma_{nt'_1}\omega\rp(x)- \fr1n h\lp n x, n(t'_3-t'_2);  h\lp n x, n(t'_2-t'_1); \vph^n_g, \gamma_{nt'_1}\omega\rp, \gamma_{nt'_2}\omega\rp(x)\right|\\
&\leq \fr{2C_0}n+ \fr{4+C_0}n+0 =\fr{4+3C_0}n
}
where in the second inequality we used \eqref{snh} on the first two terms, \eqref{hlamb} and Proposition~\ref{prop:local} on the third term, and \eqref{semi} on the fourth term. We also know that $S_n(s,t;g',\omega)\in \Gamma$ for any $s,t\in \fr1n \mbb N$ and $g'\in \Gamma$, so
\al{
\|S_n(t'_1,t'_3;g,\omega)-S_n(t'_2,t'_3;S_n(t'_1,t'_2;g,\omega),\omega)\|_\infty&\leq \fr{4+3C_0+2}n= \fr{6+3C_0}n.
}
Again choosing $n$ large enough, we can make sure $(6+3C_0)/n<\vep/2$. 
\end{proof}
Lemma~\ref{f2eclosed} and Lemma~\ref{f2ebign} together imply that $\mu(E_2^\vep)=1$. Since $\vep>0$ can be arbitrarily small, we conclude that $\mu$-almost surely, $\forall t_1,t_2,t_3$ such that $0\leq t_1\leq t_2\leq t_3\leq T$, 
\al{D(S(t_1,t_3;\cdot), S(t_2,t_3;S(t_1,t_2;\cdot)))=0,}
which implies Condition~\eqref{cond2}
\al{
S(t_1,t_3;g)=S(t_2,t_3;S(t_1,t_2;g)),\;\;\;\forall g\in \Gamma.
} 

\iffalse{
Next we consider condition~\ref{cond3}. For any $n\in\mbb Z_{>0}$, let $s'=\fr1n\lfloor sn\rfloor$ and $t'=\fr1n\lfloor tn\rfloor$. Then $|(t-s)-(t'-s')|\leq \fr2n$. Combining \eqref{ptlip}, Proposition~\ref{prop:error} and \eqref{sndif}, we deduce that, $\forall x\in \fr1n\mbb Z^2$,
\al{
\left|\wh S(0,t-s;g,\omega)(x)-\fr1nh\lp nx,n(t'-s');\vph_g^n,\omega\rp\right|&\leq \fr{2(10+2C_0)+18+2C_0+C_0}n=\fr{38+7C_0}n.
}
On the other hand, for the same reason,
\al{
\left|\wh S(s,t;g,\omega)(x)-\fr1nh\lp nx,n(t'-s');\vph_g^n,\gamma_{ns'}\omega\rp\right|&\leq \fr{38+7C_0}n.
}

The case for condition~\ref{cond4} is similar.}\fi

For Condition~\eqref{cond3}, define $E_3^\vep$ to be the set of $S\in \ag C'_T$ with the following property: given $x\in \mbb R^2$ and $g_1,g_2\in\Gamma$, if $g_1(y)=g_2(y)$ for all $y$ such that $|y-x|_1\leq R$, then $S(s,t;g_1)(y)=S(s,t;g_2)(y)$ for all $y$ such that $|y-x|_1\leq R-2(t-s)-\vep$ and all $s,t$ such that $0\leq s< t\leq T$. Then again since uniform convergence in $\ag C'_T$ implies pointwise convergence, the set $E_3^\vep$ is closed in $\ag C'_T$. 

To show that $\wh S_n\in E_3^\vep$ for all large enough $n$, first we notice that the function $u\mapsto S_n(s,t;g-u/n, \omega)+u/n:\mbb R\rar \Gamma$ has a period of 4 by the construction of $S_n$. Therefore
\al{
\wh S_n(s,t;g,\omega)=\fr 14\int_{u=0}^{ 4}\lb S_n\lp s,t; g-g(x)-\fr u n,\omega\rp+\fr un\rb du +g(x).
}
Now applying Lemma~\ref{lem:snlinear}, we deduce that $\wh S_n\in E_3^\vep$ as long as $\fr9n<\vep$, so $\mu(E_3^\vep)=1$. Since $\vep>0$ is arbitrary, we conclude that Condition~\eqref{cond3} holds $\mu$-almost surely.
\iffalse{
Finally, we consider condition~\ref{cond8}. }\fi
\end{proof}

Following the same proof procedure as Proposition~\ref{prop:local} and Lemma~\ref{lem:local}, Condition~\eqref{cond3} implies the following localization property about the limit points.
\begin{lem}\label{lem:limitlocal}
Suppose $S\in \ag C'_T$ satisfies Condition~\eqref{cond3}, then given $g_1,g_2\in \Gamma$ and $x\in\mbb R^2$,
\al{
\sup_{y:|y-x|_1\leq R-2(t-s)}\left|S(s,t;g_1)(y)-S(s,t;g_2)(y)\right|&\leq \sup_{y:|y-x|_1\leq R}\left|g_1(y)-g_2(y)\right|.
}
\end{lem}

\section{Equilibrium measures}\label{sec:em}
\subsection{Construction of Gibbs measures}\label{sec:gibbs}

We will make use of a particular family of Gibbs measures of dimer coverings in the plane. Specifically, for each slope vector $\rho:=(\rho_1,\rho_2)\in U^o$, the interior of the Newton polygon, we would like a Gibbs measure whose height function in large scale is concentrated on the slope $\rho$. One such choice is to restrict the Boltzmann measure on the toroidal graph $G_n$ to those with height change $(\lfloor n\rho_1\rfloor,\lfloor n\rho_2\rfloor)$, and take any weak limit as $n\rar \infty$. However, it is unclear how to compute the limiting local probabilities.

A different approach, following \cite{RK97,CKP01,KOS06}, is to consider the full Boltzmann measure on $G'_n$, modified from $G_n$ by gauge transformations of the edge weights while keeping the periodicity. Then conditioned on the dimer configuration outside a finite region, the measure inside the finite region is unchanged, so a limiting Gibbs measure is still a Gibbs measure of the original graph $G$. However, the absolute probability of a dimer covering on $G_n$ is modified according to its height change, with certain height changes preferred to others. Then by saddle point analysis, the limiting Gibbs measure is concentrated on a preferred slope. The advantage of this approach is that the local dimer probabilities of the Boltzmann measure on tori are computable from the relevant entries of the inverse Kasteleyn matrices, where a Kasteleyn matrix is a weighted adjacency matrix with certain choice of signs. The Kasteleyn matrices on $G'_n$ can be inverted via explicit diagonalization, and it can be shown that, as $n\rar \infty$, the inverse matrix converges along a common subsequence to a limiting matrix, which is called the infinite inverse Kasteleyn matrix. (In fact, the whole sequence converges due to the uniqueness result of Gibbs measures by Sheffield(\cite{SS03}), but this uniqueness is not needed in this paper.) We choose our Gibbs measure to be the weak limit along this subsequence, where the local dimer probabilities of the Gibbs measure are computed from the relevant entries of the infinite inverse Kasteleyn matrix in the same way as on tori.

In our specific example as well as more general examples mentioned in Section~\ref{sec:comm}, where the dimer graph and the edge weights satisfy an ``isoradiality'' condition, it is shown in \cite{BdT07} based on \cite{RK02} that the entries of the infinite inverse Kasteleyn matrix have  simple expressions based on local geometry.  As a result, explicit formula for the local dimer probabilities in the Gibbs measure can be derived.

We shall state the relevant results for our specific example. For each even face $x$,  denote the four edges of the face $x$ on the north, east, south and west side by $n_x,e_x,s_x,w_x$ respectively. The gauge transformation on $G_n$ and $G$ is such that the dimer weights are the same in every even face $x$ (this is true on the original graphs). The isoradiality condition is that the four edges can be represented by the four sides of a quadilateral with unit circumcircle as in Figure~\ref{fig:isoradial}, such that the dimer weights of $n_x,e_x,s_x,w_x$ are given by $\sin\ap, \sin\beta, \sin\gamma,\sin\delta$ respectively.
\begin{figure}[h]
\caption{}
\includegraphics[scale=1.3]{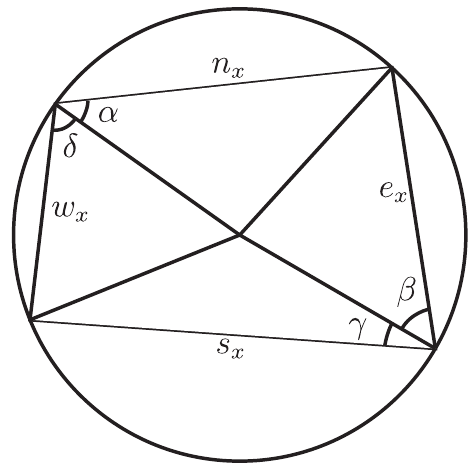}
\centering
\label{fig:isoradial}
\end{figure}

Then the Gibbs measure has the following local dimer probabilities, where $P($some edges$)$ is shorthand for the probability that the set of edges are included simultaneously,
\ba{
&P(n_x)=\fr{\ap}{\pi}, \;\;P(e_x)=\fr{\beta}{\pi}, \;\;P(s_x)=\fr{\gamma}{\pi}, \;\;P(w_x)=\fr{\delta}\pi,\\
&P(n_x, s_x)=\fr1{\pi^2}\lp \ap\gamma+\fr{\sin \ap\sin\gamma}{\sin \beta\sin\delta}\beta\delta\rp,\\
&P(e_x, w_x)=\fr1{\pi^2}\lp \beta\delta+\fr{\sin \beta\sin\delta}{\sin \ap\sin\gamma}\ap\gamma\rp=aP(n_x, s_x).
}

The slope $\rho=(\rho_1,\rho_2)$ is given by the expected height change along $x$ and $y$ directions, so
\ba{
&\rho_1=2\lp P(e_x)-P(w_x)\rp=2\fr{\beta-\delta}\pi\label{e63}\\
&\rho_2=2\lp P(s_x)-P(n_x) \rp=2\fr{\gamma-\ap}\pi.\label{e64}
}
Also, it can be easily checked that, this new set of weights being a gauge transformation of the original graph is equivalent to that
\ba{
\fr{\sin\beta\sin\delta}{\sin\ap\sin\gamma}=a,\label{e65}
}
which also guarantees that the shuffling dynamics is identical with the new weights. Then we can compute $\ap, \beta,\gamma, \delta$ as functions of $\rho$, similar to \cite{CKP01}. Using $\ap+\beta+\gamma+\delta=\pi$, we can rewrite \eqref{e65} as
\al{
\cos(\beta-\delta)-\cos(\beta+\delta)=a\lp\cos(\ap-\gamma)+\cos(\beta+\delta)\rp,
}
and plugging in \eqref{e63} and \eqref{e64} to get
\al{
\fr{1}{1+a}\cos\lp\fr{\pi \rho_1}2\rp-\fr{a}{1+a}\cos\lp\fr{\pi \rho_2}2\rp=\cos\lp\beta+\delta\rp.
}
Denoting the LHS by $M$ and combining this with \eqref{e63}, we get
\al{
\beta&=\fr{\pi\rho_1}4+\fr12\cos^{-1}\lp M\rp, \;\;\delta=-\fr{\pi\rho_1}4+\fr12\cos^{-1}\lp M\rp.
}
Similarly, we obtain that
\al{
\ap&=-\fr{\pi\rho_2}4+\fr12\cos^{-1}\lp-M\rp,\;\;\gamma=\fr{\pi\rho_2}4+\fr12\cos^{-1}\lp-M\rp.
}

We shall verify that the slope of the Gibbs measure does concentrate on $\rho$. Let $\Omega_\rho$, where $\rho\in U^o$,  denote the probability space of height functions $h(\cdot)$ distributed according to the Gibbs measure $\pi_\rho$ constructed above with slope $\rho$, and shifted vertically so that $h(0)=0$ always holds.

\begin{lem}\label{lem:flat}
Suppose $h(\cdot)$ is given by $\Omega_\rho$. For all $R>0$,
\al{
\lim_{n\rar \infty} \mbb{E} \sup_{|x|_1\leq R}\left| \fr1nh(\lfloor nx\rfloor)-x\cdot \rho\right|=0.
}
\end{lem}
\begin{proof}
As a sanity check, notice that $\mbb E(h(\lfloor nx\rfloor))-\lfloor nx\rfloor\cdot \rho$ is bounded. By  \cite{KOS06}, the variance of $h(x)$ is $O(\log |x|)$. Therefore, given some small $\vep>0$, by Chebyshev's inequality,
\ba{
P\lp \left|h(\lfloor nx\rfloor)-\lfloor nx\rfloor\cdot \rho\right|>n\vep\rp\leq\fr {C_1\log^2|\lfloor nx\rfloor|}{n^2\vep^2}\label{ptlogbd}
}
for all large enough $n$ and some constant $C_1$ that depends on $\rho$ only. 

Pick a set of points $A_\vep\subset \lc x:|x|_1\leq R\rc$ with cardinality $O(\vep^{-2}R^2)$ such that every point $x$ with $|x|_1\leq R$ is within $\vep$ $\ell_1$-distance from at least one point in $A_\vep$. Let $M_\vep^n$ denote the event such that $|h(\lfloor nx\rfloor)-\lfloor nx\rfloor\cdot \rho|\leq n\vep$ for all $x\in A_\vep$. Then by \eqref{ptlogbd},
\ba{
P(M_\vep^n)\geq 1- \sum_{x\in A_\vep}\fr {C_1\log^2|\lfloor nx\rfloor|}{n^2\vep^2}.\label{mneprbd}
}
This is a finite sum, and when $n\rar \infty$, $n^2$ clearly outgrows $\log^2|\lfloor nx\rfloor|$, so $P(M^n_\vep)\rar 1$.

Since heights at neighboring faces differ by at most 3, if $h\in M^n_\vep$, 
\al{
\sup_{|x|_1\leq R}\left| h(\lfloor nx\rfloor)-nx\cdot \rho\right|&\leq C_2n\vep+\rho\\
\Rightarrow\sup_{|x|_1\leq R}\left| \fr1nh(\lfloor nx\rfloor)-x\cdot \rho\right|&\leq C_2\vep+\fr\rho n\label{2dbound}\numberthis
}
for some constant $C_2$ that depends on $\rho$ only.

On the other hand, $\sup_{|x|_1\leq R}\left| \fr1nh(\lfloor nx\rfloor)-x\cdot \rho\right|$ is at most $|3+\rho|R$. Combining this with \eqref{mneprbd} and \eqref{2dbound}, we see that
\al{
\limsup_{n\rar \infty} \mbb{E} \sup_{|x|_1\leq R}\left| \fr1nh(\lfloor nx\rfloor)-x\cdot \rho\right|&\leq C_2\vep.
}
Now taking $\vep\rar 0$, we obtain the statement.
\end{proof}

\subsection{Evolution at equilibrium}\label{sec:pipini}

Now we shall relate the Gibbs measures in the previous section to shuffling dynamics. The first observation is the following.
\begin{prop}
The Gibbs measure $\pi_\rho$, $\rho\in U^o$, is invariant under shuffling.
\end{prop}
\begin{proof}
This is the direct consequence of Proposition~\ref{prop:loc}. The Gibbs measure was constructed as the weak limit of a sequence of Boltzmann measures on tori with increasing sizes. Since the local moves preserve the Boltzmann measures on tori, the shuffling procedure, viewed as a sequence of local moves, also preserves the Boltzmann measures on tori. Also since the edge weights of the new graph are the same as the original graph up to gauge transformation, the shuffling procedure in fact results in the same exact Boltzmann measure as before the shuffling.

By Lemma~\ref{lem:linear} (and its version on tori), after a shuffle, the resulting  configuration on a finite region $V$ is a random function of the original configuration on the region \al{
V'=\lc x: x\text{ is within a constant }\ell_1\text{-distance of }V\rc,}where the randomness only comes from the Bernoulli random variables that govern the outcomes of the spider moves on $V'$. For each finite region, such function and the Bernoulli random variables are the same for the shuffles on all large enough tori and the infinite plane. Due to weak convergence, the measures of cylindrical sets on $V'$ and $V$ on tori converge to those in the Gibbs measure. Therefore when we perform the shuffling to the Gibbs measure in the plane, the resulting distribution on $V$ also remains the same. This is true for all $V$, and since $\pi_\rho$ is characterized by distribution on all finite regions, we conclude that $\pi_\rho$ is invariant under shuffling.
\end{proof}

The invariance of Gibbs measures will let us find the hydrodynamic limit when the height is initially distributed according to a Gibbs measure. To do so, we first investigate how the height evolution at the origin depends on the local dimer configuration. From Table~\ref{tab:1}, it is easy to see the following rules about the height at an  even face $x$  after a shuffle:
\begin{enumerate}
\item{The height remains the same when either  $s_x$ or $e_x$ is present,}
\item{The height decreases by 4 when either  $n_x$ or $s_x$ is present,}
\item{When none of  $n_x,e_x,s_x,w_x$ is present, with probability $\fr a{1+a}$ the height remains the same, and with probability $\fr 1{1+a}$ the height decreases by 4.}
\end{enumerate}

Consider the shuffling height process with initial configuration given by $\Omega_\rho$. The whole process lives on the product probability space $\Omega_\rho\times \Omega$. Define the event $Q(x,t)$ for $t\in \mbb N$ as
\al{
Q(x,t):=&\{\text{At time $t$, the even face $x$ has either $n_x$ or $s_x$ present, or,} \\
&\text{has none of $n_x, e_x, s_x, w_x$ present and $\omega(x,t)=1$}\}.
}
Recall that $\omega(x,t)=1$ with probability $\fr 1{1+a}$. We can use the information gathered in Section~\ref{sec:gibbs} and the invariance of the Gibbs measure under shuffling to compute explictly $P(Q(x,t))$:
\al{
P(Q(x,t))=&P_{\pi_\rho}(n_x)+P_{\pi_\rho}(s_x)- P_{\pi_\rho}(n_x, s_x)+\fr 1{1+a}(1-P_{\pi_\rho}(e_x)-P_{\pi_\rho}(w_x)+ P_{\pi_\rho}(e_x,w_x)\\
&-P_{\pi_\rho}(n_x)-P_{\pi_\rho}(s_x)+ P_{\pi_\rho}(n_x, s_x))\\
=&P_{\pi_\rho}(n_x)+P_{\pi_\rho}(s_x)- \fr a{1+a}P_{\pi_\rho}(n_x, s_x)+\fr 1{1+a}P_{\pi_\rho}(e_x, w_x)\\
=&P_{\pi_\rho}(n_x)+P_{\pi_\rho}(s_x)\\
=&\fr1{\pi}\cos^{-1}\lp \fr{a}{1+a}\cos\lp\fr{\pi \rho_2}2\rp - \fr{1}{1+a}\cos\lp\fr{\pi \rho_1}2\rp\rp.
}

\iffalse

\beta&=-\fr{\pi\rho_1}4+\fr12\cos^{-1}\lp M\rp, \;\;\delta=\fr{\pi\rho_1}4+\fr12\cos^{-1}\lp M\rp.
}
Similarly, we obtain that
\al{
\ap&=\fr{\pi\rho_2}4+\fr12\cos^{-1}\lp-M\rp,\;\;\gamma=-\fr{\pi\rho_2}4+\fr12\cos^{-1}\lp-M\rp.
\fi

The result is a function of $\rho$ only. We define $H(\rho):=4P(Q(x,t))$. 

Now we are ready to prove a law of large numbers for the height evolution at the origin. Finer fluctuation results could be obtained as in \cite{CT18}.
\begin{lem}\label{lem:originspeed}
Suppose the height process $h(x,t)$ has an initial condition $h(\cdot, 0)$ given by $\Omega_\rho$, $\rho\in U^o$, then
\al{\lim_{t\rar \infty}\mbb{E}\left|\fr1th(0,t)+H(\rho)\right|=0.
}
\end{lem}
\begin{proof}
Since $\fr1t h(0,t)$ is bounded between $0$ and $4$, to prove the lemma, it suffices to show that for every $\vep>0$,
\ba{
\lim_{t\rar \infty} P\lp \left|\fr1th(0,t)+H(\rho)\right|>\vep\rp=0.\label{want63}
}
By the local rules above, we know that for every even face $x$, and $t\in \mbb N$, \ba{
\mbb E(h(x,t)-h(x,0))=\mbb E\lp-4\sum_{i=1}^t \mbf 1_{Q(x,t)}(h)\rp=-tH(\rho).\label{632}
} In particular, $\mbb E(h(0,t))=-tH(\rho)$. The intuition is that, if $\left|\fr1th(0,t)+H(\rho)\right|>\vep$, then faces near origin also have large deviation simultaneously. Since at each $t$, the correlation of $Q(x,t)$ for different $x$  decays at least quadratically in distance by \cite{KOS06}, this event is unlikely to happen.

For convenience, we will work on faces $2x$ where $x\in \mbb Z^2$, i.e. $T$-translations of the origin. They have the extra benefit that $\mbb E_{\Omega_\rho}(h(2x))=2x\cdot \rho$ exactly (otherwise there is a bounded error). For $R\in \mbb N$, let \al{V_R(t)&:=\var\lp \sum_{|x|_1= R} (h(2x,t)-h(2x,0)) \rp, \wt V_R(t):=\var\lp\sum_{|x|_1= R} \mbf 1_{Q(2x,t)}(h)\rp,} then by the local rules
\al{
V_R(t+1)=&\var\lp \sum_{|x|_1= R} (h(2x,t)-h(2x,0))-4\sum_{|x|_1= R} \mbf 1_{Q(2x,t)}(h)\rp\\
=& V_R(t)-8\cov\lp\sum_{|x|_1= R} (h(2x,t)-h(2x,0)), \sum_{|x|_1= R} \mbf 1_{Q(2x,t)}(h)\rp+16\wt V_R(t)
}
Using Cauchy-Schwarz inequality, the covariance is bounded in absolute value by $\sqrt{V_R(t)\wt V_R(t)}$. As mentioned above, $|\cov(Q(x,t),Q(y,t))|=O(|x-y|^2_1)$, so $\wt V_R(t)=O(R)$. Therefore $V_R(t)$ satisfies the recurrence inequality
\al{
|V_R(t+1)-V_R(t)|\leq a\sqrt{R}\sqrt{V_R(t)}+bR
}
for some constants $a,b$, which implies $V_R(t)=O(t^2R)$.

From now on we let $R=R(\vep, t):=\lfloor\vep t/20\rfloor$, then $V_R(t)=O(\vep t^3)$. By Chebyshev's inequaility, 
\ba{
P\lp\left|\sum_{|x|_1=R} (h(2x,t)-h(2x,0))-\mbb E\sum_{|x|_1= R} (h(2x,t)-h(2x,0))\right|\geq C\vep^2t^2\rp=\fr{C'}{\vep^3t}\label{633}
}
which goes to 0 as $t\rar \infty$.

By Lemma~\ref{lem:flat}, given $\delta>0$, for all large enough $t$,
\ba{
P\lp\sup_{|x|_1= R}\left|\fr1th(2x,0)-\mbb E\lp\fr1th(2x,0)\rp\right|\leq\fr \vep{10}\rp >1-\delta.\label{634}
}
If $\fr1th(0,t)+H(\rho)>\vep$, since height functions restricted to even faces are 2-spatially-Lipschitz,  we will have, for all $x$ such that $|x|_1=R$, \al{\fr1th(2x,t)-\mbb E\lp \fr1th(2x,t)\rp&\geq \fr1th(0,t)+H(\rho)-\left|\fr1th(2x,t)-\fr1th(0,t)\right|-\left|\mbb E\lp \fr1th(2x,t)-\fr1th(0,t)\rp\right|\\
&\geq\vep-2\fr \vep{10}-2\fr \vep{ 10}=\fr35\vep.}  If this event along with the event in \eqref{634} both happen, then the event in \eqref{633} happens, because in this case for all $x$ such that $|x|_1=R$ (which has cardinality $\Omega(\vep t)$), 
\al{
 (h(2x,t)-h(2x,0))-\mbb E (h(2x,t)-h(2x,0))\geq t\lp\fr35\vep-\fr1{10}\vep\rp=\fr{\vep t}{2}.
}
Therefore 
\al{\lim_{t\rar \infty}P\lp\fr1th(0,t)+H(\rho)>\vep\rp\leq \delta.}
Similarly, we can deduce that
\al{
\lim_{t\rar \infty}P\lp\fr1th(0,t)+H(\rho)<-\vep\rp\leq \delta.
}
Taking $\delta\rar 0$, we proved \eqref{want63}.
\end{proof}

\begin{prop}\label{prop:futureflat}
With the same assumption as Lemma~\ref{lem:originspeed}, for all $R>0$ and $t>0$, 
\al{
\lim_{n\rar \infty} \mbb{E}\sup_{|x|_1\leq R}\left| \fr1nh(\lfloor nx\rfloor,\lfloor nt\rfloor)-x\cdot \rho+tH(\rho)\right|=0.
}
\end{prop}
\begin{proof}
On one hand, by Lemma~\ref{lem:originspeed}, 
\ba{
\lim_{n\rar \infty} \mbb{E}\left|\fr1nh(0,\lfloor nt\rfloor)+tH(\rho)\right|&=t\lim_{n\rar \infty} \mbb{E}\left|\fr1{nt}h(0,\lfloor nt\rfloor)+H(\rho)\right|=0.\label{ff1}
}
On the other hand, since the shuffling process preserves the Gibbs measure of dimer coverings, the random function $h(\cdot, \lfloor nt\rfloor)-h(0,\lfloor nt\rfloor)$ is also distributed according to $\Omega_\rho$. Therefore, by Lemma~\ref{lem:flat}, we know that
\ba{
\lim_{n\rar \infty} \mbb{E} \sup_{|x|_1\leq R}\left| \fr1n\lp h(\lfloor nx\rfloor,\lfloor nt\rfloor)-h(0,\lfloor nt\rfloor)\rp-x\cdot \rho\right|=0.\label{ff2}
}
Now we combine \eqref{ff1} and \eqref{ff2} to deduce that
\al{
&\lim_{n\rar \infty} \mbb{E} \sup_{|x|_1\leq R}\left| \fr1nh(\lfloor nx\rfloor,\lfloor nt\rfloor)-x\cdot \rho+tH(\rho)\right|\\
\leq &\lim_{n\rar \infty} \mbb{E}\lp\sup_{|x|_1\leq R}\left| \fr1n\lp h(\lfloor nx\rfloor,\lfloor nt\rfloor)-h(0,\lfloor nt\rfloor)\rp-x\cdot \rho\right|+ \left|\fr1nh(0,\lfloor nt\rfloor)+tH(\rho)\right|\rp\\
=&0+0=0.
}
\end{proof}

\subsection{More on limit points}

It turns out that the information we obtained from the Gibbs measures tells us more about the limit points in Section~\ref{sec:limit}. First, we need a lemma that relates the initial conditions in Section~\ref{sec:smoothspace} and Section~\ref{sec:pipini}.

Recall from Section~\ref{sec:hydro} that the sequence of height processes $(h_n(x,t))_{n\in \mbb Z_{>0}}$ has its initial conditions $h_n(\cdot, 0)$ given by the probability space $\Omega_0$. 
\begin{lem}\label{lem:initapprox}
Given $g\in \Gamma$, if
\ba{
\lim_{n\rar 0}\mbb{E}\sup_{|x|_1\leq R} \left| \fr1nh_n(\lfloor nx\rfloor,0)-g(x)\right|=0\label{iaass}
}
for all $R>0$, then
\ba{
\lim_{n\rar 0}\mbb{E}\sup_{|x|_1\leq R, t\leq T} \left| \fr1nh_n(\lfloor nx\rfloor,\lfloor nt\rfloor)-
\wh S_n(0,t;g)(x)\right|=0\label{iastt}
}
for all $R>0$.
\end{lem}
\begin{proof}
By \eqref{g2hbd}, we might as well replace the assumption \eqref{iaass} by
\ba{\label{iaass2}
\lim_{n\rar 0}\mbb{E}\sup_{|x|_1\leq R} \fr1n\left| h_n(\lfloor nx\rfloor,0)-\vph_g^n(\lfloor nx\rfloor)\right|=0,\;\;\;\;\forall R>0.
}

To be precise about the source of the randomness, we let $\omega_0\in\Omega_0$ refer to the initial condition, and $\omega\in \Omega$ refer to the Bernoulli mark as usual that dictates the shuffling. Using \eqref{sndif} and Proposition~\ref{prop:error}, we know that, for a given $t\geq 0$,
\ba{
&\sup_{|x|_1\leq R} \left| \fr1nh_n(\lfloor nx\rfloor,\lfloor nt\rfloor;\omega_0;\omega)-
\wh S_n(0,t;g,\omega)\right|\nonumber\\
\leq & \sup_{|x|_1\leq R} \fr1n\left| h_n(\lfloor nx\rfloor,\lfloor nt\rfloor;\omega_0;\omega)-h\lp \lfloor nx\rfloor,\lfloor nt\rfloor;\vph_g^n, \omega\rp\right|+\fr{C_3}n\label{ia1}
}
for some global constant $C_3$. And by Proposition~\ref{prop:local},
\al{
&\sup_{|x|_1\leq R} \fr1n\left| h_n(\lfloor nx\rfloor,\lfloor nt\rfloor;\omega_0;\omega)-h\lp \lfloor nx\rfloor,\lfloor nt\rfloor;\vph_g^n, \omega\rp\right|\\\leq& \sup_{|x|_1\leq R+2t+1} \fr1n\left| h_n(\lfloor nx\rfloor,0;\omega_0)-\vph_g^n(\lfloor nx\rfloor)\right|\numberthis\label{ia2}
}
for all $n$ large enough. Therefore, combining \eqref{ia1} and \eqref{ia2},
\al{
&\limsup_{n\rar \infty}\mbb{E}\sup_{|x|_1\leq R, t\leq T} \left| \fr1nh_n(\lfloor nx\rfloor,\lfloor nt\rfloor)-\wh S_n(0,t;g)(x)\right|\\
= &\limsup_{n\rar \infty}\int \sup_{|x|_1\leq R, t\leq T} \left| \fr1nh_n(\lfloor nx\rfloor,\lfloor nt\rfloor;\omega_0;\omega)-\wh S_n(0,t;g,\omega)(x)\right|d\omega_0 d\omega\\
\leq &\limsup_{n\rar \infty}\int\sup_{|x|_1\leq R+2T+1}\lp\fr1n\left| h_n(\lfloor nx\rfloor,0;\omega_0)-\vph_g^n(\lfloor nx\rfloor)\right|+\fr{C_3}n\rp d\omega_0 \\
=&\lim_{n\rar \infty}\mbb{E}\sup_{|x|_1\leq R+2T+1} \fr1n\left| h_n(\lfloor nx\rfloor,0)-\vph_g^n(\lfloor nx\rfloor)\right|+\lim_{n\rar 0}\fr{C_3}n=0
}
by our assumption \eqref{iaass2}.
\end{proof}

Now we can make the following additional characterization about the limit points of $\lp \mu_n\rp_{n\in\mbb Z_{>0}}$. 
\begin{prop}\label{prop:limitflat}
All subsequential limits of $\lp \mu_n\rp_{n\in\mbb Z_{>0}}$ satisfy the following property almost surely:
\al{
S(s,t;g_\rho)=g_\rho-(t-s)H(\rho),\;\;\;\;\forall \rho\in U, 0\leq s \leq t\leq T
}
where $g_\rho(x)=x\cdot \rho$.
\end{prop}
\begin{proof}
Suppose a subsequence $(\mu_i)_{i\in\mbb N}:=(\mu_{n_i})_{i\in\mbb N}$ converges weakly to $\mu$. First let us fix some $R>0$, $\rho\in U^o$ and $t\in(0,T]$. Combining Lemma~\ref{lem:flat}, Proposition~\ref{prop:futureflat} and Lemma~\ref{lem:initapprox}, we obtain that
\ba{
\lim_{i\rar \infty} \mbb{E}\sup_{|x|_1\leq R}\left|\wh S_{n_i}(0,t;g_\rho)(x)-g_\rho(x)+tH(\rho)\right|=0.
}

Viewing $\dps\sup_{|x|_1\leq R}\left|S(0,t;g_\rho)(x)-g_\rho(x)+tH(\rho)\right|$ as a function of $S\in \ag C'_T$, it is clearly continuous, tracing through the definitions \eqref{ctmetric} \eqref{Ddef} of the metric for the uniform topology. It is also bounded because $S(0,t;\cdot)\in \ag E_q^r$. Therefore, by the definition of weak convergence, 
\ba{
\mbb{E}_\mu\sup_{|x|_1\leq R}\left| S(0,t;g_\rho)(x)-g_\rho(x)+tH(\rho)\right|=\lim_{i\rar \infty} \mbb{E}\sup_{|x|_1\leq R}\left|\wh S_{n_i}(0,t;g_\rho)(x)-g_\rho(x)+tH(\rho)\right|=0.
}
Since $R>0$ is arbitrary, we deduce that $S(0,t;g_\rho)=g_\rho-tH(\rho)$ $\mu$-almost surely.

We can choose a dense subset of $t\in (0,T]$ and $\rho\in U^o$ so that $S(0,t;g_\rho)=g_\rho-tH(\rho)$ holds $\mu$-almost surely for all such $t$ and $\rho$.

Now because $H(\rho)$ is in fact continuous on the entire $U$, by definition \eqref{gammametric}, $d(g_\rho-tH(\rho),g_{\rho'}-tH(\rho'))\rar 0$ as $\rho'\rar \rho$. Then by Lemma~\ref{lem:lipg}, we know that, fixing $t$, $S(0,t;g_\rho)=g_\rho-tH(\rho)$ for all $\rho\in U$ $\mu$-almost surely.

Recall that we are working on $\ag C'_T$, so $S(0,t;\cdot)$ is $(10+2C_0)$-Lipschitz continuous in $t$ $\mu$-almost surely. Again due to uniform topology, we conclude that $S(0,t;g_\rho)=g_\rho-tH(\rho)$ holds for all $t\in [0,T]$ and $\rho\in U$ $\mu$-almost surely.

We already know from Proposition~\ref{prop:char} that Condition~\eqref{cond2} holds $\mu$-almost surely. Then $\mu$-almost surely, for all $s,t$ such that $0\leq s\leq t\leq T$, using Property~\eqref{prop1},
\al{
S(s,t;g_\rho)&=S(s,t;g_\rho-sH(\rho))+sH(\rho)=S(s,t;S(0,s;g_\rho))+sH(\rho)\\
&=S(0,t;g_\rho)+sH(\rho)=g_\rho-(t-s)H(\rho),
}
so the proposition is proved.
\end{proof}

\section{Viscosity solution}\label{sec:vs}

We recall some PDE theory about Hamilton-Jacobi equations. For more details, see \cite{EL10}. Given $g, H\in C^0(\mbb R^2)$, consider the following first-order partial differential equation with initial condition
$g$
\begin{equation}
\lc\begin{array}{rl}u_t+H(u_x)&=0\\
u(x,0)&=g(x)\end{array}\right.\label{hje}
\end{equation}
where the solution $u(x,t)$ is a function on $\mbb R^2\times [0,T]$, and $u_x$ is supposed to be its gradient with respect to the two spatial coordinates.  It is possible to apply method of characteristics to obtain short-time solution, but even if $g$ and $H$ are smooth, shocks can form at finite time and the solution $u(x,t)$ becomes nondifferentiable. In order to describe the long-time evolution of the PDE and to deal with nonsmooth initial conditions, we have to consider weak solutions, which are not differentiable but still solve the PDE in some sense. A priori, the uniqueness and existence of weak solutions are not guaranteed. The viscosity solution is a particular choice that guarantee both. A function $u(x,t)$  is called a viscosity solution to \eqref{hje} if the following  conditions hold,
\begin{enumerate}
\item{$u$ is continuous on $\mbb R^2\times [0,T]$,}
\item{$u(\cdot, 0)= g$,}
\item{If $\phi\in C^\infty(\mbb R^2\times [0,T])$ and $(x_0, t_0)\in  \mbb R^2\times (0,T)$ satisfy that $\phi(x_0, t_0)= u(x_0, t_0)$ and $\phi\geq u$ in a neighborhood of $(x_0, t_0)$, then
\ba{\phi_t(x_0,t_0)+H(\phi_x(x_0, t_0))\leq 0,
\label{localmax}}}
\item{If $\phi\in C^\infty(\mbb R^2\times [0,T])$ and $(x_0, t_0)\in  \mbb R^2\times (0,T)$ satisfy that $\phi(x_0, t_0)= u(x_0, t_0)$ and $\phi\leq u$ in a neighborhood of $(x_0, t_0)$, then
\ba{\phi_t(x_0,t_0)+H(\phi_x(x_0, t_0))\geq 0.
\label{localmin}}
}

\end{enumerate}

To prove the main result, we first identify the semigroup of the shuffling height process with the semigroup of the PDE.
\begin{prop}\label{prop:idsemi}
All subsequential limits of $(\mu_n)_{n\in \mbb Z>0}$ are concentrated on a single $S\in \ag C'_T$, such that $u(x,t):=S(0,t;g)(x)$ coincides with the unique viscosity solution to \eqref{pde}.
\end{prop}
\begin{proof}
We know that all subsequential limits of $(\mu_n)_{n\in \mbb Z>0}$ have probability 1 on $S\in \ag C'_T$ satisfying Condition~\eqref{cond1}, \ref{cond2} and \ref{cond3} as well as the property in Proposition~\ref{prop:limitflat}. 

We will check the four criteria for the viscosity solution. The continuity of $u$ is guaranteed by the fact that $S(0,t;g)\in \Gamma$ and the Lipschitz continuity of $S$ in $t$. Also Condition~\eqref{cond1} implies that $u(\cdot, 0)=g$. We shall verify criterion \eqref{localmax}. The last one \eqref{localmin} is similar.

Let $\phi$ and $(x_0,t_0)$ be as described in the assumption for \eqref{localmax}. Suppose $B\subset \mbb R^2\times (0,T)$ is an open ball centered at $(x_0, t_0)$ in which $\phi\geq u$. An issue here  is that {\it $\phi(\cdot, t)$ is not necessarily in $\Gamma$}, so we cannot directly apply $S$ on $\phi(\cdot, t)$. First, we prove the following lemma.
\begin{lem}\label{lem:phider}
$
\phi_x(x_0,t_0)\in U.
$
\end{lem}
\begin{proof}
Let $\mbf v:=\phi_x(x_0,t_0)$, and define vector $\mbf w:=(w_1,w_2)$. By definition, for all small $k\in\mbb R$,
\ba{
\phi\lp x_0+k\mbf w,t_0\rp=\phi\lp x_0, t_0\rp+k\mbf v\cdot \mbf w+o(k).\label{711}
}
Since $\phi(x_0,t_0)=u(x_0,t_0)$ and $\phi\geq u$ in $B$, for small $k$,
\ba{
\phi\lp x_0+k\mbf w,t_0\rp-\phi\lp x_0, t_0\rp&\geq u\lp x_0+k\mbf w,t_0\rp-u\lp x_0, t_0\rp\geq -2|k||\mbf w|_\infty.\label{712}
}
Combining \eqref{711} and \eqref{712}, we get for all small $k$,
\al{
k\mbf v\cdot \mbf w+o(k)&\geq -2|k||\mbf w|_\infty.
}
Divided by $k$, this implies 
\al{
-2|\mbf w|_\infty\leq \mbf v\cdot \mbf w\leq 2|\mbf w|_\infty.
}
Taking $\mbf w=(\pm1, \pm1)$, we obtain a set of inequalities of the form $\pm v_1\pm v_2\leq 2$. Therefore $\mbf v$ must satisfy that $|\mbf v|_1\leq 2$, or $\mbf v\in U$.
\end{proof}

Now we define a new function affine in space
\al{
\wt \phi(x,t):=\phi(x_0,t)+(x-x_0)\cdot \phi_x(x_0,t_0).
}
We have $\wt \phi(x_0,t_0)=\phi(x_0,t_0)=u(x_0,t_0)$. Since $\phi$ is smooth, we can assume that for $(x,t)\in B$,
\al{
\phi(x,t)-\wt \phi(x,t)&=\phi(x,t)-\phi(x_0,t)-(x-x_0)\cdot \phi_x(x_0,t_0)\\
&=(x-x_0)\cdot \phi_x(x_0,t)+O((x-x_0)^2)-(x-x_0)\cdot \phi_x(x_0,t_0)\\
&=O((x-x_0)^2+(x-x_0)(t-t_0)).
}
As $\phi\geq u$ in $B$, we have 
\ba{
\wt \phi(x,t)-u(x,t)\geq -C((x-x_0)^2+(x-x_0)(t-t_0))\label{newphiu}
}
for some constant $C>0$ and all $(x,t)\in B$. Furthermore $\wt \phi_x(x,t)=\phi_x(x_0,t_0)$, and $\wt \phi_t(x,t)=\phi_t(x_0,t)$. In particular, $\wt \phi(\cdot, t)\in \Gamma$ for all $t$.

For all small enough $\delta$, we have
\al{
\lc (x,t):|x-x_0|_1\leq 2\delta, |t-t_0|_1\leq \delta\rc\subset B.
}
By Condition~\eqref{cond2},
\al{
\wt \phi(x_0,t_0)&=u(x_0,t_0)=S(t_0-\delta, t_0;S(0,t_0-\delta;g))(x_0)=S(t_0-\delta, t_0;u(\cdot, t_0-\delta))(x_0).
}
Define $g_1=u(\cdot, t_0-\delta)$ and $g_2=u(\cdot, t_0-\delta)\wedge \wt \phi(\cdot, t_0-\delta)$ ($g_2\in \Gamma$ as shown in the proof of Lemma~\ref{lem:local}). From \eqref{newphiu}, we have
\al{
\sup_{y:|y-x_0|_1\leq 2\delta}|g_1(y)-g_2(y)|\leq C\delta^2,
}
Now we apply Lemma~\ref{lem:limitlocal} to $g_1$ and $g_2$ to deduce
\al{
|S(t_0-\delta, t_0;g_1)(x_0)-S(t_0-\delta, t_0;g_2)(x_0)|\leq C\delta^2.
}
Since $\wt \phi(\cdot, t_0-\delta)\geq g_2$, by Property~\eqref{prop2}, 
\al{
S(t_0-\delta, t_0;\wt \phi(\cdot, t_0-\delta))(x_0)&\geq S(t_0-\delta, t_0;g_2)(x_0)\\
&\geq S(t_0-\delta, t_0;g_1)(x_0)-C\delta^2=\wt \phi(x_0,t_0)-C\delta^2.
}

On the other hand, using Property~\eqref{prop1} and the property in Proposition~\ref{prop:limitflat},
\al{
S(t_0-\delta, t_0;\wt \phi(\cdot, t_0-\delta))(x_0)&=S(t_0-\delta, t_0;\wt \phi(\cdot, t_0-\delta)- \wt \phi(0,t_0-\delta))(x_0)+ \wt \phi(0,t_0-\delta)\\
&=S(t_0-\delta, t_0;g_{\phi_x(x_0,t_0)})(x_0)+\wt \phi(0,t_0-\delta)\\
&=g_{\phi_x(x_0,t_0)}(x_0)-\delta H(\phi_x(x_0,t_0))+\wt \phi(0,t_0-\delta)\\
&=\wt \phi(x_0, t_0-\delta)- \wt \phi(0,t_0-\delta)-\delta H(\phi_x(x_0,t_0))+\wt \phi(0,t_0-\delta)\\
&=\wt \phi(x_0, t_0-\delta)-\delta H(\phi_x(x_0,t_0))
}
where we recall $g_\rho:=x\cdot \rho$ for $\rho\in \mbb R^2$. Along with the inequality above, we get
\al{
\wt \phi(x_0, t_0-\delta)-\delta H(\phi_x(x_0,t_0))&\geq \wt \phi(x_0,t_0)-C\delta^2\\
\Rightarrow \fr{\wt \phi(x_0,t_0)- \wt \phi(x_0, t_0-\delta)}\delta+H(\phi_x(x_0,t_0))&\leq C\delta.
}
Since $\wt \phi_t(x_0,t_0)=\phi_t(x_0,t_0)$, by taking $\delta\rar 0$, we obtain that
\al{
\phi_t(x_0,t_0)+H(\phi_x(x_0,t_0))\leq 0,
}
exactly as desired.

Even though $H$ is not defined on the whole $\mbb R^2$, notice that we only used the values of $H$ on $U$, so we still have the uniqueness of $u(x,t)$ with initial condition $g\in \Gamma$.

\end{proof}

The same argument shows that in fact $S(s,t;g)$ is determined for all $(s,t)\in [0,T]^2$ and $g\in \Gamma$.

\begin{proof}[Proof of Theorem~\ref{thm:main}]

Let $u$ be the unique viscosity solution to \eqref{pde} with initial condition $g$. From Proposition~\ref{prop:idsemi} and the precompactness of $\lc \mu_n\rc$, we deduce that in fact the entire sequence of random variables $\wh S_n$ converges weakly to the deterministic $\wh S\in \ag C'_T$ characterized by $\wh S(s,t;g)=u(0,t-s)$.

Tracing through the definition of the metric \eqref{ctmetric}, \eqref{Ddef} and \eqref{gammak}, it is easy to see that the following function is continuous on $\ag C'_T$,
\al{
f(S):=\sup_{|x|_1\leq T, t\leq T} \left|S(0,t;g)(x)-u(x,t)\right|.
}
It is also bounded because $S(0,t;g)(x)$ is bounded when $|x|_1\leq T, t\leq T$, due to Property~\eqref{prop3} and the fact that $g\in \Gamma$. Therefore
\al{
\lim_{n\rar \infty}\mbb E\sup_{|x|_1\leq T, t\leq T} \left|\wh S_n(0,t;g)(x)-u(x,t)\right|
&=\lim_{n\rar \infty}\mbb E f(\wh S_n)=\mbb Ef(\wh S)\\
&=\mbb E\sup_{|x|_1\leq T, t\leq T} \left|\wh S(0,t;g)(x)-u(x,t)\right|=0.\numberthis\label{whslimit}
}
From  \eqref{sndif} and Proposition~\ref{prop:error}, we know that
\ba{
\lim_{n\rar \infty}\mbb E\sup_{x\in \mbb R^2, t\leq T} \left|\wh S_n(0,t;g)(x)-\fr1nh\lp\lfloor nx\rfloor, \lfloor nt\rfloor;\vph_g^n\rp\right|=0.\label{whsnheight}
}
Also by \eqref{g2hbd},
\ba{
\lim_{n\rar \infty}\sup_{x\in \mbb R^2}\left|\fr1n\vph_g^n(\lfloor nx\rfloor)-g(x)\right|=0.\label{vphg}
}

Now consider a sequence of shuffling height processes $(h_n(x,t))_{n\in \mbb Z_{>0}}$ with initial conditions given by a probability space $\Omega_0$, satisfying \eqref{init} for every $R>0$. Then \eqref{hydro} follows with $R=T$, by putting together \eqref{whslimit}, \eqref{whsnheight}, \eqref{vphg} and Lemma~\ref{lem:initapprox}. Finally we can choose $T>0$ arbitrarily, so Theorem~\ref{thm:main} is established.
\end{proof}

\section{Remarks}\label{sec:comm}

The more general dimer shuffling height process introduced in Section~\ref{sec:dimershuffle} encompasses  the shuffling dynamics on the 2-periodic $\mbb Z^2$ lattice considered in \cite{CT18}, as well as shuffling built from resistor networks discussed in \cite{AK11}. We shall briefly explain the construction of the latter.

Consider a toroidal weighted graph $G$, not necessarily bipartite. A zig-zag path on the graph is a path that alternatingly turns maximally left and right. The graph is minimal if, when lifted to the universal cover, zig-zag paths do not self-intersect, and two zig-zag paths intersect at most once. By repeatedly performing the $Y$-$\Delta$ local moves, one can modify the graph in the way that, combinatorially, one specific zig-zag path is slided once around the torus while the other zig-zag paths stay put. In general, the edge weights will be different after such an operation. However, if the graph is isoradial, similar to Section~\ref{sec:gibbs} except that the edge weight is given by tangent instead of sine of the angle, the $Y$-$\Delta$ moves can be performed while keeping isoradiality as well as the transversal directions of the zig-zag paths.  As a result, after a sliding operation, the edge weights return to the original ones. This can be turned into a dimer shuffling by the Temperley bijection, and each $Y$-$\Delta$ move can be decomposed into four spider moves (see \cite[Lemma 5.11]{AK11}).

In terms of extending the results in this paper, the case of the 2-periodic $\mbb Z^2$ lattice can be done similarly using the results from \cite{CT18}. In the case of a more complicated graph as above, the results in Section~\ref{sec:em} can be extended without much effort, and the general strategy should still work, but the approximation schemes seem to be more complicated and graph-dependent, and beyond the scope of this paper. 

\begin{appendices}
\section{Proof of Proposition~\ref{prop:loc}}

The statement for vertex contraction/expansion move is obvious, because the mapping does not change the total weight of the dimer covering.

To prove the statement for the spider move, we refer to Figure~\ref{fig:spider} and \ref{fig:spidershuffle}. By an abuse of notation, we use $a,b,c,d, A,B,C,D$ to denote both the edges and their weights. If $S$ is a set of edges in the inner square on the LHS, then let $w^T_S$ be the weight of the dimer covering where exactly $S$ among the four edges of the inner square are included and $T$ is the rest of the edges in the covering. Define $\wt w^T_S$ similarly for the weight of the dimer covering on the RHS formed exactly by $S$, $T$ and a subset of the four new tentacles, where $S$ is a set of edges in the inner square and $T$ is a set of edges in the complement of the inner square and the four tentacles. Now consider the three rows in Figure~\ref{fig:spidershuffle} and the three omitted rows. We shall first verify that the ratio between the LHS and RHS of every row in Figure~\ref{fig:spidershuffle} is the same. By \eqref{spider},
\al{
&\fr{\wt w^T_{\{AC\}}+\wt w^T_{\{BD\}}}{w^T_{\emptyset}}=AC+BD=\fr{ac+bd}{(ac+bd)^2}=\fr1{ac+bd},\\
&\fr{\wt w^T_{\{A\}}}{w^T_{\{c\}}}=\fr{\wt w^T_{\{B\}}}{w^T_{\{d\}}}=\fr{\wt w^T_{\{C\}}}{w^T_{\{a\}}}=\fr{\wt w^T_{\{D\}}}{w^T_{\{b\}}}=\fr1{ac+bd},\\
&\fr{\wt w^T_{\emptyset}}{w^T_{\{ac\}}+w^T_{\{bd\}}}=\fr1{ac+bd}.
}
Therefore, if the dimer coverings were distributed according to their weights before the spider move, the probabilities after the spider move are also proportional to their weights, except in the first row of Figure~\ref{fig:spidershuffle}, the two results on the RHS are grouped together. It only remains to check that their probabilities are also proportional to their weights. Indeed,
\al{
\fr{\wt w^T_{\{AC\}}}{\wt w^T_{\{BD\}}}=\fr{AC}{BD},
}
and the spider move selects $A$ and $C$ with probability $\fr{AC}{AC+BD}$ and selects $B$ and $D$ with probability $\fr{BD}{AC+BD}$.

\section{Local interpolations for constructing $S_n$}\label{appenb}

Suppose we know the values of the function $f:=\psi_{s,t}$ at, for example, the four faces $(0,0), (2,0), (2,2), (0,2)$, and want to interpolate $f$ inside the square formed by these four points. Up to symmetry and a vertical shift, we shall assume that $f(0,0)=0$, $f$ is nonnegative at the other three points, and $f(2,0)\leq f(0,2)$. We consider all possible cases.

{\it Case 1}: If $f(2,0)=f(0,2)=f(2,2)=0$, we let $f=0$ inside the square.

{\it Case 2}: If $f(2,0)=f(0,2)=0$, and $f(2,2)=4$, for $(x,y)\in [0,2]^2$, we let $f(x,y)=2y$ if $x\geq y$, and $f(x,y)=2x$ if $x<y$. 

{\it Case 3}: If $f(2,0)=0$, $f(0,2)=4$ and $f(2,2)=0$, this is same as Case 2 under rotation.

{\it Case 4}: If $f(2,0)=0$, $f(0,2)=4$ and $f(2,2)=4$, define $f(x,y)=2y$.

{\it Case 5}: If $f(2,0)=f(0,2)=4$ and $f(2,2)=0$, for $(x,y)\in[0,2]^2$, we let $f(x,y)=2x$ for $y\leq \min\{2-x,x\}$, $f(x,y)=4-2y$ for $2-x\leq y\leq x$, $f(x,y)=4-2x$ for $y\geq \max\{2-x,x\}$, and $f(x,y)=2y$ for $x\leq y\leq 2-x$.

{\it Case 6}: If $f(2,0)=f(0,2)=f(2,2)=4$, this is an inverted version of Case 2.

Notice that in all the cases, the interpolation is linear on the four edges of the $2\times 2$ square. Therefore, we can glue together the interpolations on all the $2\times 2$ squares to obtain a global 2-spatially Lipschitz interpolation.

\end{appendices}

\section*{Acknowledgement}

The author would like to thank Richard Kenyon for the introduction to dimer theory and many patient discussions. The author also wants to thank Fabio Toninelli and Sanjay Ramassamy for giving helpful advice and pointing to references. Finally, the author thanks Kavita Ramanan and Vincent Lerouvillois for pointing out errors in the draft.

\bibliographystyle{icmcta4}
\bibliography{mybib}

\begin{thebibliography}{10}

\bibitem{BP13}
Billingsley, P.: Convergence of probability measures.
\newblock John Wiley \& Sons (2013)

\bibitem{BF14}
Borodin, A., Ferrari, P.L.: Anisotropic growth of random surfaces in $2+ 1$
  dimensions.
\newblock Communications in Mathematical Physics \textbf{325}(2), 603--684
  (2014)

\bibitem{BF15}
Borodin, A., Ferrari, P.L.: Random tilings and {Markov} chains for interlacing
  particles  (2015).
\newblock \href {http://arxiv.org/abs/1506.03910} {\path{arXiv:1506.03910}}

\bibitem{CT18}
Chhita, S., Toninelli, F.L.: A (2+1)-dimensional anisotropic {KPZ} growth model
  with a rigid phase  (2018).
\newblock \href {http://arxiv.org/abs/1802.05493} {\path{arXiv:1802.05493}}

\bibitem{CKP01}
Cohn, H., Kenyon, R., Propp, J.: A variational principle for domino tilings.
\newblock Journal of the American Mathematical Society \textbf{14}(2), 297--346
  (2001)

\bibitem{EKLP92}
Elkies, N., Kuperberg, G., Larsen, M., Propp, J.: Alternating-sign matrices and
  domino tilings (part {II}).
\newblock Journal of Algebraic Combinatorics \textbf{1}(3), 219--234 (1992)

\bibitem{EL10}
Evans, L.C.: {Partial Differential Equations}.
\newblock American Mathematical Society, 2nd edn. (2010)

\bibitem{LE14}
Evans, L.C.: Envelopes and nonconvex {Hamilton--Jacobi} equations.
\newblock Calculus of Variations and Partial Differential Equations
  \textbf{50}(1-2), 257--282 (2014)

\bibitem{AK11}
Goncharov, A.B., Kenyon, R.: Dimers and cluster integrable systems.
\newblock Ann. Sci. Éc. Norm. Supér. (to appear) (2011).
\newblock \href {http://arxiv.org/abs/1107.5588v2 [math.AG]}
  {\path{arXiv:1107.5588v2 [math.AG]}}

\bibitem{JPS98}
Jockusch, W., Propp, J., Shor, P.: Random domino tilings and the {Arctic Circle
  Theorem}  (1998).
\newblock \href {http://arxiv.org/abs/math/9801068v1 [math.CO]}
  {\path{arXiv:math/9801068v1 [math.CO]}}

\bibitem{KJ05}
Johansson, K.: The {Arctic Circle} boundary and the {Airy} process.
\newblock Annals of Probability \textbf{33}(1), 1--30 (2005)

\bibitem{AK12}
Kechris, A.: Classical descriptive set theory, vol. 156.
\newblock Springer Science \& Business Media (2012)

\bibitem{KJ17}
Kelley, J.L.: General topology.
\newblock Courier Dover Publications (2017)

\bibitem{RK97}
Kenyon, R.: Local statistics of lattice dimers.
\newblock Annales de l'Institut Henri Poincare (B) Probability and Statistics
  \textbf{33}(5), 591--618 (1997)

\bibitem{RK02}
Kenyon, R.: The {Laplacian and Dirac} operators on critical planar graphs.
\newblock Inventiones Mathematicae \textbf{150}(2), 409--439 (2002)

\bibitem{RK09}
Kenyon, R.: Lectures on dimers.
\newblock In: Statistical mechanics, IAS/Park City Math. Ser., 16, pp.
  191--230. Amer. Math. Soc. (2009)

\bibitem{KOS06}
Kenyon, R., Okounkov, A., Sheffield, S.: Dimers and amoebae.
\newblock Annals of Mathematics pp. 1019--1056 (2006)

\bibitem{KC13}
Kipnis, C., Landim, C.: Scaling limits of interacting particle systems, vol.
  320.
\newblock Springer Science \& Business Media (2013)

\bibitem{LT17}
Legras, M., Toninelli, F.L.: Hydrodynamic limit and viscosity solutions for a
  {2D} growth process in the anisotropic {KPZ} class  (2017).
\newblock \href {http://arxiv.org/abs/1704.06581} {\path{arXiv:1704.06581}}

\bibitem{EN10}
Nordenstam, E.: On the shuffling algorithm for domino tilings.
\newblock Electronic Journal of Probability \textbf{15}, 75--95 (2010)

\bibitem{JP03}
Propp, J.: Generalized domino-shuffling.
\newblock Theoretical Computer Science \textbf{303}(2-3), 267--301 (2003)

\bibitem{FR01}
Rezakhanlou, F.: Continuum limit for some growth models {II}.
\newblock Annals of Probability pp. 1329--1372 (2001)

\bibitem{FR02}
Rezakhanlou, F.: Continuum limit for some growth models.
\newblock Stochastic processes and their applications \textbf{101}(1), 1--41
  (2002)

\bibitem{TS99}
Sepp{\"a}l{\"a}inen, T.: Existence of hydrodynamics for the totally asymmetric
  simple {K-exclusion} process.
\newblock Annals of Probability \textbf{27}(1), 361--415 (1999)

\bibitem{SS03}
Sheffield, S.: Random surfaces: large deviations principles and gradient gibbs
  measure classifications.
\newblock Ph.D. thesis, Stanford University (2003)

\bibitem{DS07}
Speyer, D.E.: Perfect matchings and the octahedron recurrence.
\newblock Journal of Algebraic Combinatorics \textbf{25}(3), 309--348 (2007)

\bibitem{WT90}
Thurston, W.P.: Conway's tiling groups.
\newblock The American Mathematical Monthly \textbf{97}(8), 757--773 (1990)

\bibitem{BdT07}
de~Tili{\`e}re, B.: Quadri-tilings of the plane.
\newblock Probability Theory and Related Fields \textbf{137}(3-4), 487--518
  (2007)

\bibitem{FT17}
Toninelli, F.L.: A $(2+ 1) $-dimensional growth process with explicit
  stationary measures.
\newblock Annals of Probability \textbf{45}(5), 2899--2940 (2017)

\end{thebibliography}

\end{document}